\documentclass[12pt]{amsart}
\usepackage{amsmath}
\usepackage{amsfonts,comment}
\usepackage{amssymb,color}
\usepackage[hidelinks]{hyperref}
\usepackage{amsthm,doi}
\usepackage{thmtools}
\usepackage{physics}
\usepackage{tikz}
\usetikzlibrary{decorations.pathreplacing}
\usepackage{algorithm}
\usepackage{algorithmic}
\usepackage{soul,enumitem}
\usepackage[showonlyrefs]{mathtools}

\usepackage{wasysym}

\usepackage{subcaption}
\mathtoolsset{showonlyrefs=true}

\makeatletter
\newcommand\RedeclareMathOperator{%
	\@ifstar{\def\rmo@s{m}\rmo@redeclare}{\def\rmo@s{o}\rmo@redeclare}%
}

\newcommand\rmo@redeclare[2]{%
	\begingroup \escapechar\m@ne\xdef\@gtempa{{\string#1}}\endgroup
	\expandafter\@ifundefined\@gtempa
	{\@latex@error{\noexpand#1undefined}\@ehc}%
	\relax
	\expandafter\rmo@declmathop\rmo@s{#1}{#2}}

\newcommand\rmo@declmathop[3]{%
	\DeclareRobustCommand{#2}{\qopname\newmcodes@#1{#3}}%
}
\@onlypreamble\RedeclareMathOperator
\makeatother

\RedeclareMathOperator{\var}{Var}
\DeclareMathOperator{\cov}{Cov}
\DeclareMathOperator{\sgn}{sgn}

\newcommand{\C}{\mathbb{C}}

\newcommand{\charge}{\kappa}

\newcommand{\E}{\mathbb{E}}

\newcommand{\cons}{\eta}

\newcommand{\cgrid}{\Delta}
\newcommand{\cdelta}{\delta^{*}}
\newcommand{\ccdelta}{\delta^{**}}
\newcommand{\grid}{\Lambda}

\newcommand{\gridlplus}{\Lambda_{L+\ccdelta}}
\newcommand{\gridsparse}{X_{L}}
\newcommand{\gridsparseno}{X}
\newcommand{\gridsparseplus}{X_{L+1}}
\newcommand{\domain}{\Omega_L}
\newcommand{\domainplus}{\Omega_{L+\ccdelta}}

\newcommand{\domainminus}{\Omega_{L-\ccdelta}}

\newcommand{\outset}{\mathrm{Z}}
\newcommand{\outseta}{\mathrm{Z}_1}
\newcommand{\outsetb}{\mathrm{Z}_2}
\newcommand{\map}{\Phi}

\newcommand{\alstep}[1]{\noindent {\bf #1}}

\newcommand{\sdelta}{\lceil\delta^{-1/2}\rceil}
\newcommand{\sdeltaa}{\lceil \frac14 \delta^{-1/2}\rceil}
\newcommand{\Dop}{\mathcal{D}}

\newcommand{\algo}{PhaseJumps}
\newcommand{\talgo}{Twisted PhaseJumps}
\newcommand{\calgo}{Twisted PhaseJumps-coarse}

\newtheorem{lemma}{Lemma}[section]
\newtheorem{theorem}[lemma]{Theorem}

\newtheorem{prop}[lemma]{Proposition}

\theoremstyle{definition}

\numberwithin{equation}{section}
\newtheorem{rem}[lemma]{Remark}

\allowdisplaybreaks

\author[A. Haimi]{Antti Haimi}
\address[A. Haimi]{Faculty of Science and Engineering, Åbo Akademi University,
Tuomiokirkontori 3
20500 Turku, Finland}
\email{antti.haimi@abo.fi}
\author[G. Koliander]{G\"{u}nther Koliander}
\address[G. Koliander]{Acoustics Research Institute, Austrian Academy of Sciences,
Dr. Ignaz Seipel-Platz 2,	AT-1010 Vienna, Austria}
\email{gkoliander@kfs.oeaw.ac.at}
\author[J. L. Romero]{Jos\'{e} Luis Romero}
\address[J. L. Romero]{Faculty of Mathematics, University of Vienna, Oskar-Morgenstern-Platz 1, A-1090 Vienna, Austria, and Acoustics Research Institute, Austrian Academy of Sciences, Dr. Ignaz Seipel-Platz 2,	AT-1010 Vienna, Austria}
\email{jose.luis.romero@univie.ac.at}

\thanks{This research was funded in whole or in part by the Austrian Science Fund (FWF): 10.55776/Y1199. For open access purposes, the authors have applied a CC BY public copyright license to any author-accepted manuscript version arising from this submission.}

\keywords{Short-time Fourier transform, spectrogram, zeros set, charge, critical point, random function, Wasserstein metric, 
Gaussian Weyl-Heisenberg Function}

\subjclass[2020]{60G15, 60G55, 94A12, 42A61}

\title[PhaseJumps: fast computation of zeros from planar grid samples]{PhaseJumps: fast computation of zeros from planar grid samples}

\begin{document}
\begin{abstract}

We consider complex-valued functions on the complex plane and the task of computing their zeros from samples taken along a finite grid. We introduce PhaseJumps, an algorithm based on comparing changes in the complex phase and local oscillations among neighboring grid points. The algorithm is applicable to possibly non-analytic input functions, and also computes the direction of phase winding around zeros.

PhaseJumps provides a first effective means to compute the zeros of the short-time Fourier transform of an analog signal with respect to a general analysis window, 
and makes certain recent signal processing insights more widely applicable, overcoming previous constraints to analytic transformations. We study the performance of (a variant of) PhaseJumps under a stochastic input model motivated by signal processing applications and show that the input instances that may cause the algorithm to fail are fragile, in the sense that they are regularized by additive noise (smoothed analysis). Precisely, given samples of a function on a grid with spacing $\delta$, we show that our algorithm computes zeros with accuracy $\sqrt{\delta}$ in the Wasserstein metric with failure probability $O\big(\log^2(\tfrac{1}{\delta}) \delta\big)$, while numerical experiments suggest even better performance.
\end{abstract}

\maketitle

\section{Introduction}

\subsection{Computation of zero sets from grid values}\label{sec_1}

Let $F\colon \mathbb{C} \to \mathbb{C}$ be a continuous function and consider the problem of computing an approximation of its zero set within the box
\begin{align}\label{eq_Omega}
\Omega_L =\{x+iy\,:\, |x|,|y| \leq L\}.
\end{align}
We assume that we have access to the values of $F$ only on those points from the grid
\begin{align}\label{eq_Lambda}
\Lambda = \{ \delta k + i \delta j : k,j \in \mathbb{Z} \}
\end{align}
that lie inside (or very near) the \emph{computation domain} $\Omega_L$. We think of the inverse of the grid spacing $\delta$ as the \emph{resolution of the data}. 

When $F$ is analytic, the Minimal Grid Neighbors (MGN) algorithm provides a fast and highly effective way to carry out the desired computation. The method consists in collecting all grid points $\lambda \in \Lambda \cap \Omega_L$ for which the weighted magnitude $e^{-|\lambda|^2/2}|F(\lambda)|$ is minimal among all the immediate grid neighbors of $\lambda$:
\begin{align}\label{eq_intro_test}
e^{-|\lambda|^2/2}|F(\lambda)| \leq 
e^{-|\lambda'|^2/2}|F(\lambda')|, \qquad
\lambda' \in \Lambda, \quad |\lambda'-\lambda|_\infty=\delta.
\end{align}
(This requires access to $F$ not only on the grid points $\Lambda \cap \Omega_L$ but on the slightly larger set of all their immediate grid neighbors.)
The MGN algorithm was implicitly introduced in \cite{flandrin2015time} and has since become an important tool in signal processing, where the \emph{short-time Fourier transform} of a signal is given on a grid, and its zeros are used as landmarks that help distinguish signal components from noise (see Section \ref{sec_stft}). While no algorithm can work for all possible input functions (as the values of $F$ on a finite grid do not determine $F$ elsewhere), the remarkable effectiveness of (a slight variant of) MGN has been theoretically validated under a stochastic input model aimed at describing its performance in practice \cite{efkr24, smooth}.

Unfortunately, the MGN algorithm does not work well with non-analytic input functions $F$, because the success of the minimality test \eqref{eq_intro_test} in finding zeros depends crucially on the superharmonicity of the weighted function $e^{-|z|^2/2} |F(z)|$ (see \cite[Section 8.2.2]{gafbook} and \cite[Lemma 3.1]{efkr24}). This restriction limits the applicability of filtering methods based on short-time Fourier transform zeros, as many time-frequency representations used in practice are non-analytic.

In this article we introduce and analyze a new algorithm, called \algo, that provides fast computation of zero sets of possibly non-analytic functions from their values on a finite grid. As we discuss in Section \ref{sec_stft}, this provides a first means to compute the zeros of general time-frequency representations.

\subsection{The PhaseJumps algorithm}
Consider again the problem of computing the zero set $\{F=0\} \cap \Omega_L$ of a possibly non-analytic function $F$. The \algo\, algorithm identifies certain points on the finite grid
\begin{align*}
\grid_L := \grid \cap \Omega_L =
\left\{ \delta k+i \delta j: k,j \in \mathbb{Z}, |\delta k|, |\delta j| \leq L \right\}
\end{align*}
as numerical approximations of zeros, 
by detecting changes of the complex argument (phase) around those points. The key parameter of the algorithm is a choice of \emph{phase-stabilizing factors}
\begin{align}
\nu(z,w) \in \mathbb{C} \setminus \{0\}, \qquad z,w\in\mathbb{C},
\end{align}
which are only required to depend differentiably on $(z,w)$ (in the real sense). Once the factors are fixed, the corresponding \emph{phase-stabilized shifts} are
\begin{align}\label{eq_ps}
F_w(z)= \nu(z,w) F(z+w), \qquad z,w \in \mathbb{C}.
\end{align}
The default choice for the stabilizing factors is $\nu(z,w)=1$, in which case \eqref{eq_ps} are the standard Euclidean translations. Later, we will choose $\nu(z,w)=e^{-i \Im(z \bar w)}$ and show that these factors are effective in applications to time-frequency analysis; see Figure \ref{fig:phase_detail} (left). We think of $F \mapsto F_w$ as a centering operation; independently of the particular choice of the (non-vanishing) phase-stabilizing factors, the phase winding of $F_w(z)$ near $z=0$ coincides with that of $F(z)$ near $z=w$.

The \algo\, algorithm determines which boxes of half side-length $2\delta$ centered at $\lambda \in \grid_L$ contain zeros of $F$ by computing an approximation of the argument change along the boundary of the box. Specifically, we enumerate all points $\mu_j$ of the grid $\Lambda$ with $|\mu_j|_\infty=2\delta$ anti-clockwise and keep $\lambda$ if
\begin{align}\label{eq_intro_theta}
\sum_{j=1}^N \arg \big[\,F_\lambda(\mu_{j}) \overline{F_\lambda(\mu_{j-1})}\,\big] \not=0.
\end{align}
Here $\arg\colon \C\setminus\{0\} \to (-\pi,\pi]$ is the principal branch of the argument, which we further define for convenience as $\arg(0) := 0$. Of course, this approximate computation of the argument change may be incorrect if the argument of $F$ changes very significantly between grid neighbors. We take this possibility into account by eliminating $\lambda$ as a candidate zero if the maximal argument fluctuation of $F$ along neighboring points in the boundary of the test box is larger than $90\%$ of the total possible argument change. More precisely, we only keep $\lambda$ if, in addition to \eqref{eq_intro_theta}, the following holds:
\begin{align}\label{eq_intro_theta_2}
\max_{j=1, \dots, N} \big| \arg \big[\,F_\lambda(\mu_{j}) \overline{F_\lambda(\mu_{j-1})}\big] \big|< 0.9 \cdot \pi.
\end{align}
The test box is taken to have half side-length $2 \delta$ (instead of $\delta/2$) because we expect the test \eqref{eq_intro_theta} to be more reliable for detecting zeros well within the box. For zeros near the boundary of the test box, the phase fluctuations may be large and \eqref{eq_intro_theta_2} may fail. We expect such a zero to be however caught by another test box centered close to it; see Figure \ref{fig:phase_detail}.

\begin{figure}[tbp]
    \centering
    \begin{subfigure}[b]{0.47\textwidth}
    \includegraphics[width=0.98\textwidth]{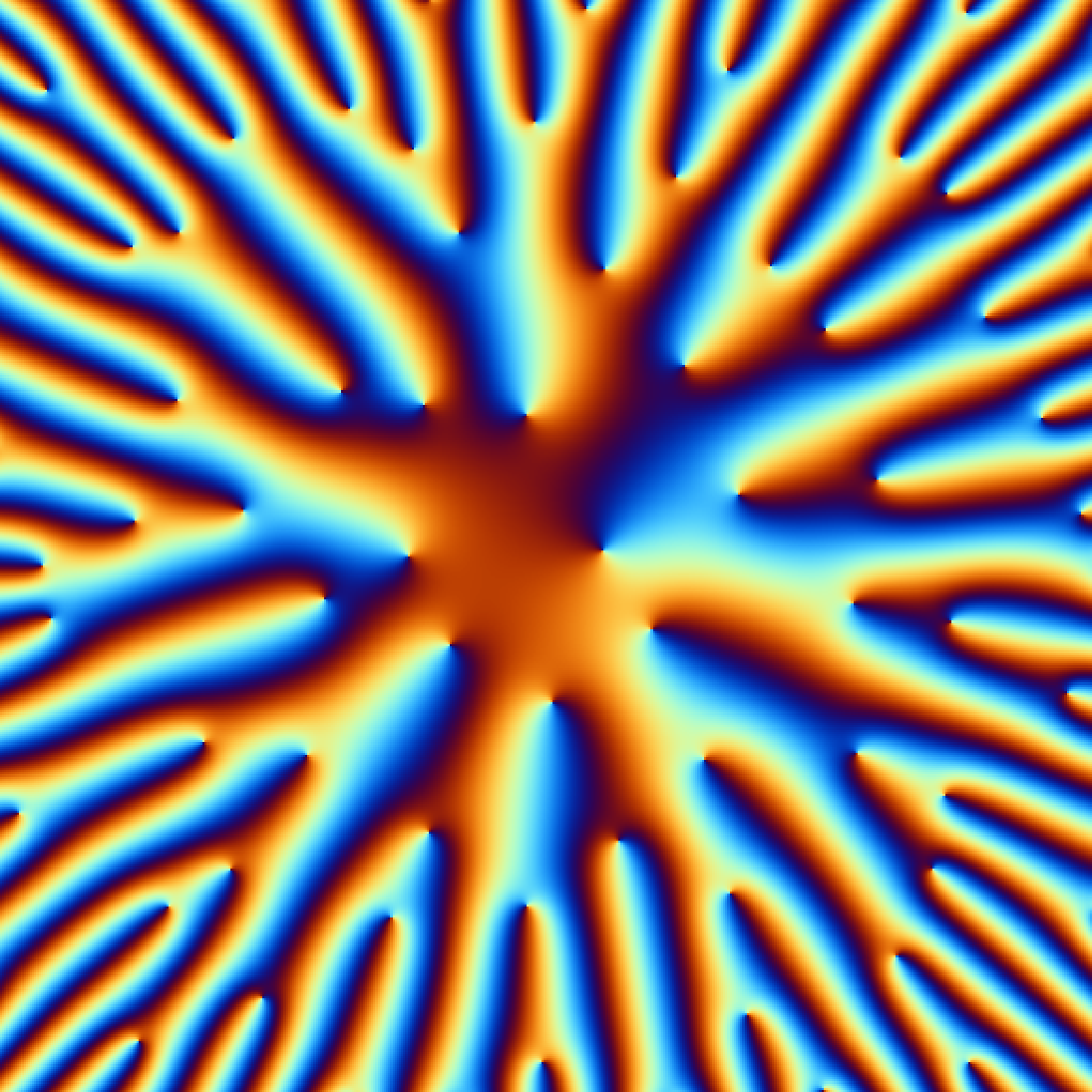}
    \end{subfigure}
    \hfill
    \begin{subfigure}[b]{0.47\textwidth}
    \includegraphics[width=0.98\textwidth]{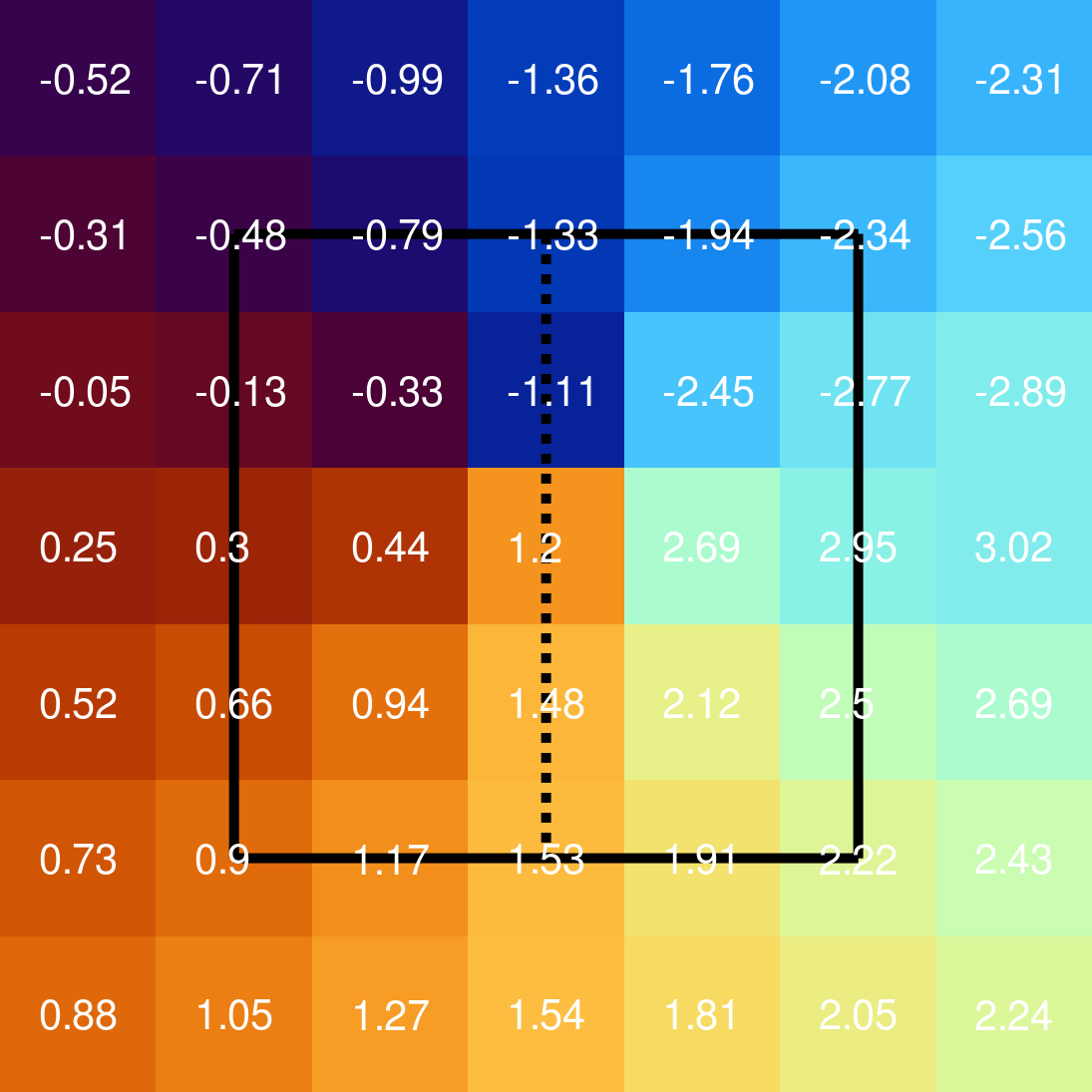}
    \end{subfigure}
    \caption{Phase values of a complex-valued function.
    \textit{Left}: In applications to time-frequency analysis (see Section \ref{sec_stft}), phases fluctuate moderately near the origin but strongly away from it. The twisted shifts \eqref{eq_ts} re-center the function while moderating phase fluctuations, as in the center of the image.
    \textit{Right}:
    A zoom onto the pixel values close to a potential zero near the center. By calculating the argument change along the black square, we can deduce that a zero is likely within the box. 
    This is justified because phase fluctuates moderately along the black box. On the other hand, along the dotted vertical segment, there is a large argument jump.}
    \label{fig:phase_detail}
\end{figure}

The disadvantage of using slightly larger test boxes is redundancy: test boxes centered at neighboring points overlap and may cause multiple detections of the same zero. We compensate for this by \emph{sieving} the set of selected points to enforce a minimal separation of $6 \delta$ between different points. The resulting set is the output of \algo, and our approximation of the zero set $\{F=0\}\cap\Omega_L$. 

Along the way, we have also computed the argument change \eqref{eq_intro_theta}, which, as we shall see, informs us about the \emph{charge of $F$} at the detected zero $z$:
\begin{align*}
\charge_z = \sgn \det DF(z) \in \{-1,0,1\},
\end{align*}
where $DF(z)$ is the differential matrix of $F\colon \mathbb{R}^2 \to \mathbb{R}^2$. The charge of $F$ at a zero $z$ describes how $F$ locally transforms orientation, and is a meaningful feature for non-analytic functions (while the conformality of analytic functions implies that their charges are always non-negative.) 

The algorithm is formally specified as follows.

\noindent\rule{\textwidth}{1pt}
\noindent {\bf Algorithm \algo: \,}\\{Compute the (charged) zero set of $F$ inside the target domain $\domain = [-L,L]^2$.}\vspace{-0.2cm}\\
\noindent\rule{\textwidth}{1pt}

\alstep{Input}: A domain length $L\geq 1$, a grid spacing parameter $\delta >0$, and samples of a function $F\colon \mathbb{C} \to \mathbb{C}$ on the grid points $\grid_{L+2\delta} = \grid \cap [-(L+2\delta),L+2\delta]^2$.

\alstep{Parameter}: A collection of phase-stabilizing factors $\{\nu(z,w)\,:\,z,w \in \mathbb{C}\} \subset \mathbb{C}\setminus\{0\}$.

\alstep{Step 1}. Set $\outseta := \Lambda \cap \domain$. For each $\lambda \in \outseta$, consider all $\mu \in \Lambda$ such that $|\mu|_\infty = 2\delta$, and order them as $\mu_1, \ldots, \mu_N$ so that $\arg(\mu_1) \leq \ldots \leq \arg(\mu_N)$.
Set $\mu_0 := \mu_N$. 

Define 
\begin{align}
	\theta_\lambda := \frac{1}{2\pi}\sum_{j=1}^N \arg \big[\,F_\lambda(\mu_{j}) \overline{F_\lambda(\mu_{j-1})}\,\big]
\end{align}
and
\begin{align}
	\chi_\lambda := \max_{j=1, \dots, N}\Big[ \tfrac{1}{\pi} \big\lvert \arg \big[\,F_\lambda(\mu_{j}) \overline{F_\lambda(\mu_{j-1})}\big] \big\rvert\Big].
\end{align}
Let $\outsetb := \{\lambda \in \outseta: \theta_\lambda \not=0 \textrm{ and } \chi_\lambda < 0.9\} $.

\alstep{Step 2}. Use an off-the-shelf clustering algorithm to select a subset $\outset \subset \outsetb$ such that elements with the same $\theta_{\lambda}$ are  $6 \delta$ separated:
\begin{align}\label{eq_6sep}
\inf \big\{ |\lambda-\lambda'|_\infty: \lambda, \lambda' \in \outset, \theta_{\lambda} = \theta_{\lambda'}, \lambda\not=\lambda'\big\} \geq 6 \delta
\end{align}
and is maximal with respect to that property, i.e., no proper superset satisfies \eqref{eq_6sep}.

\alstep{Output}: The set $\outset$ and the numbers $\{\theta_\lambda: \lambda \in \outset\}$.

\noindent(Note that the tests in Step 1 and 2 only involve grid points 
$\lambda+\mu_j \in \grid_{L+2\delta}$.)
\\
\noindent\rule{\textwidth}{1pt}
We note that \algo\, is very fast and can be used on large data sets. 

\subsection{Twisted PhaseJumps and Gabor analysis}\label{sec_stft}
In the remainder of the article, we consider $\nu(z,w)=e^{-i \Im(z \bar w)}$ as phase-stabilizing factors in \eqref{eq_ps}, which yields the \textit{twisted shifts}
\begin{align}\label{eq_ts}
F_w(z)= F(z+w) e^{-i \Im(z \bar w)}
, \qquad z,w \in \mathbb{C}.
\end{align}
We refer to the corresponding algorithm as \talgo.
This choice of phase stabilization is tailored to applications in non-stationary signal processing, as we now briefly explain.

The goal of \emph{Gabor analysis} is to analyze a so-called \emph{signal} $f\colon \mathbb{R} \to \mathbb{C}$ by means of certain discrete time-frequency measurements, known as \emph{Gabor coefficients}. These encode the correlations of $f$ with the time and frequency shifts of a reference Schwartz class \emph{window function} $g\colon \mathbb{R} \to \mathbb{C}$:
\begin{align}\label{eq_gc}
c_{k,j} = \int_{\mathbb{R}} f(t) \overline{g(t-\delta k)} e^{-2\pi i \delta j t}\, dt, \qquad k,j \in \mathbb{Z},
\end{align}
where the resolution parameter $\delta>0$ is essentially determined by the quality of the analog-to-digital conversion device.\footnote{This is of course a simplified digital acquisition model. In practice, signals are often quantized in time and then processed, which introduces other errors that we do not model with \eqref{eq_gc}.}

In essence, Gabor analysis is about \emph{sampling}, because the coefficients \eqref{eq_gc} are the values of the \emph{short-time Fourier transform of $f$}:
\begin{align}\label{eq_stft}
V_g f(x,\xi) = \int_{\mathbb{R}} f(t) \overline{g(t-x)} e^{-2\pi i \xi t}\, dt, \qquad (x,\xi) \in \mathbb{R}^2
\end{align}
on the grid \eqref{eq_Lambda}:
$c_{k,j} = V_g f(\delta k, \delta j)$. While the full collection of time-frequency correlations \eqref{eq_stft} completely determines the signal $f$, Gabor analysis aims to extract information about $f$ (or, equivalently, about $V_g f$) from the coefficients \eqref{eq_gc}, or, more realistically, from a finite subcollection thereof.

We are motivated by the problem of learning the \emph{zero set} of $V_g f$ from finite data of the form \eqref{eq_gc}. The zero set in question has been recently discovered to provide a rich set of landmarks \cite{gardner2006sparse} that help identify a non-stationary signal impacted by additive noise \cite{flandrin2015time, 7869100,MR4047541,MR4396311,MIRAMONT2024109250}. Specifically, empirical statistics for zeros -- computed with a particular signal $f$ -- are compared to ensemble averages under a certain noise model, for which precise statistics are known \cite{gafbook, MR1662451,MR4047541, bh, hkr22, MR4396311, efkr24, fhkr}. \emph{This paradigm necessitates, however, an efficient method to compute the desired zero set from finite data}, something that so far has only been available for the Gaussian window function $g(t)=e^{-\pi t^2}$. Indeed, in this special case, the change of variables
\begin{align}\label{eq_FF}
F(z) = e^{-ix\xi} \cdot V_g f(x/\sqrt{\pi},-\xi/\sqrt{\pi}), \qquad z=x+i\xi,
\end{align}
yields a weighted analytic function, and the Minimal Grid Neighbors algorithm \cite{flandrin2015time}  computes its zeros very efficiently (see Section \ref{sec_1}).\footnote{The description of the algorithm in Section \ref{sec_1} pertains to the analytic function $\tilde{F}(z)=e^{|z|^2/2} F(z)$ in lieu of $F(z)$ and thus involves the local minima of $|F(\lambda)|$.}
In contrast, for non-Gaussian window functions $g$, the short-time Fourier transform cannot be identified with a weighted analytic function \cite{MR2729877} and the MGN algorithm is not effective (see Section \ref{sec_1}). This is a shortcoming in many important applications where it is helpful to use and to average over several (non-Gaussian) window functions \cite[Section 10.2]{flandrin2018explorations}.

\subsection{Contribution}
Our contribution is two-fold. First, \algo\, provides, for the first time, an effective method to compute the zeros of the short-time Fourier transform with an arbitrary window. As an illustration, Figure
\ref{fig:sigplusnoise} shows the computation of the zeros and charges of the STFT of the signal $f(t)=e^{-t^2/2}$ embedded into white noise, with respect to the window $g(t)=t e^{-t^2/2}$; see also Section \ref{sec_num} 
for more comprehensive numerical experiments. Second, we provide the first theoretical guarantees for the computation of zeros of short-time Fourier transforms (with general windows) from grid-values. While in practice \talgo\, effectively approximates the desired zero set $\{F=0\} \cap \Omega_L$ 
of the function \eqref{eq_FF} (with general $g$) up to a resolution comparable to that of the acquisition grid \eqref{eq_Lambda}, we shall only prove that computations are accurate up to the coarser scale $\sqrt{\delta}$. To facilitate the presentation, we shall introduce a modified version of \talgo, which we call \calgo, that explicitly deals with two scales, corresponding to data and computation resolution levels.

Our work also has novel applications to spectrograms computed with Gaussian windows. While such functions can be identified with square magnitudes of analytic functions -- and therefore their zeros can be effectively computed with MGN \cite{flandrin2018explorations} -- other landmarks such as \emph{spectrogram local maxima and saddle points} are related to zeros and charges of the short-time Fourier transform with Hermite windows; see, e.g., \cite[Section 6.8]{hkr22}, \cite[Section 9.2]{fhkr}. Thus our work provides for the first time a reliable means to compute such features, which are important in relation to spectrogram reassignment \cite[Chapters 12 and 14]{flandrin2018explorations}; see Figure \ref{fig:herm_compare}.

In the same vein, by applying \algo\, to simulated inputs, we provide a first efficient means to simulate \emph{critical points of random entire functions} -- which are zeros of certain non-analytic functions, and provide heuristic models in string theory \cite{MR2104882}. Similarly, \algo\, enables the simulation of zeros of higher order random poly-entire functions \cite{hahe13, hkr22}, a task that has so far been very challenging.

\begin{figure*}[tbp]
    \centering
    \begin{subfigure}[b]{0.47\textwidth}
    \includegraphics[width=0.98\textwidth]{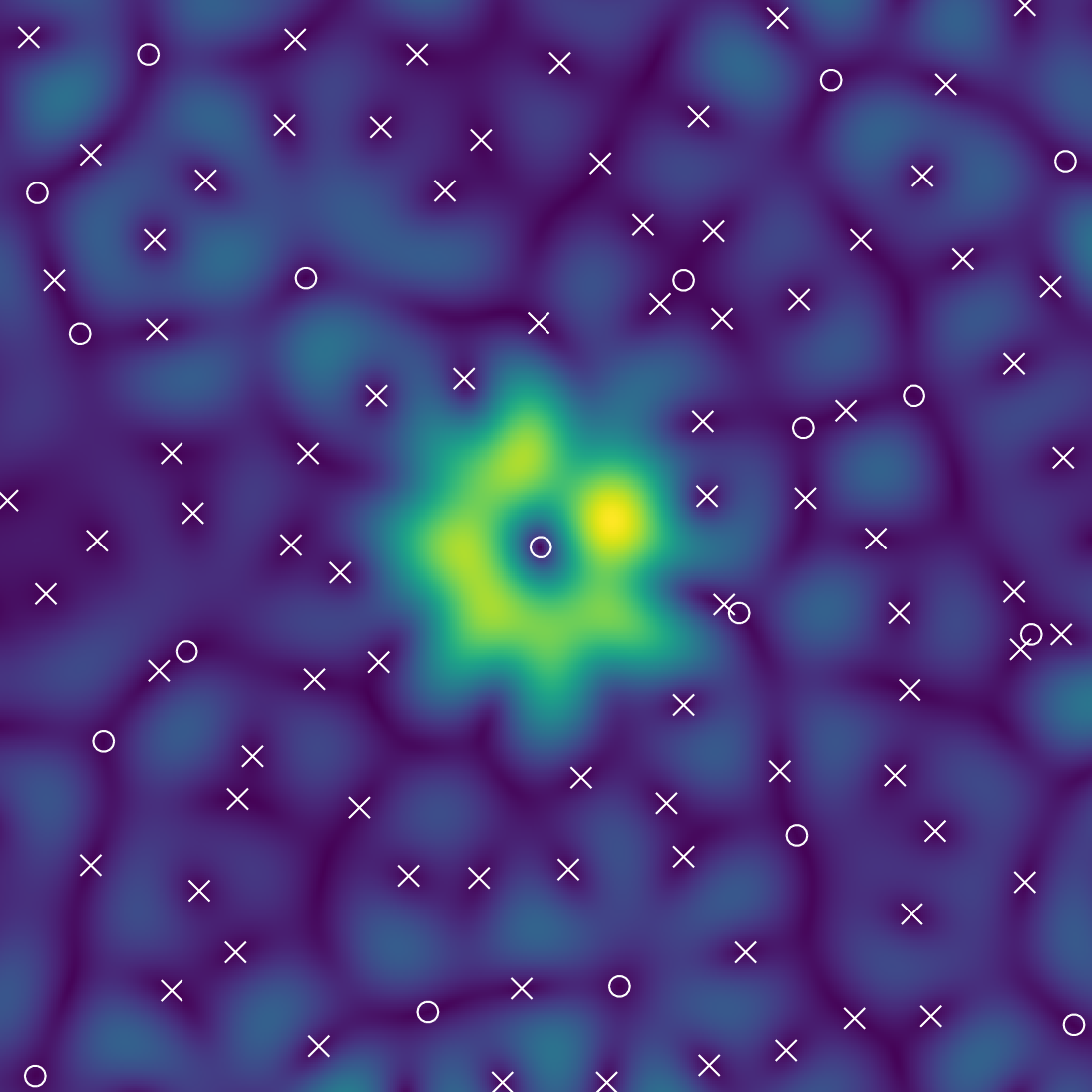}
    \end{subfigure}
    \hfill
    \begin{subfigure}[b]{0.47\textwidth}
    \centering
    \includegraphics[width=0.98\textwidth]{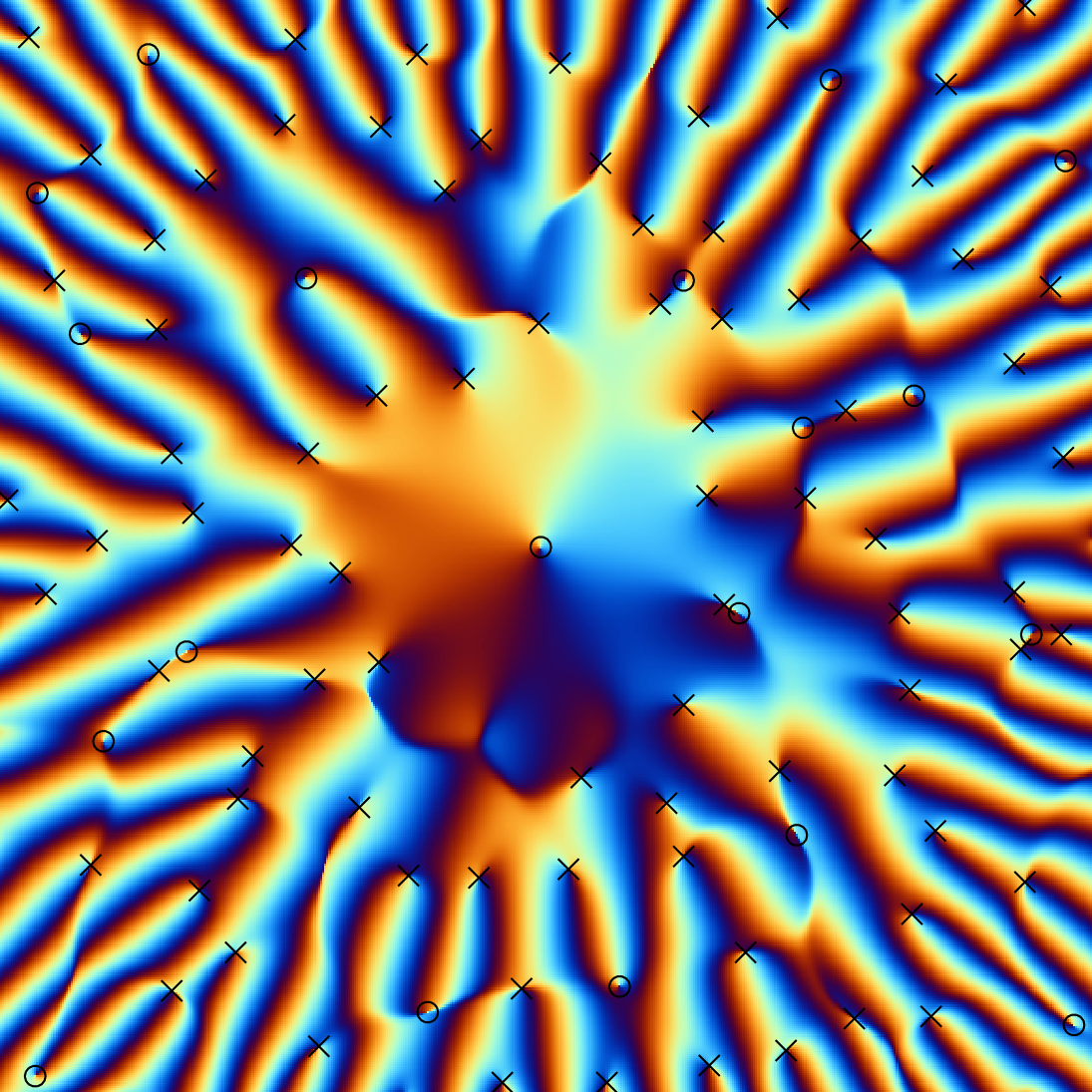}
    \end{subfigure}
    \caption{The STFT of a Gaussian signal in additive complex white noise using the Hermite function of order 1 as window 
    (\textit{Left}: Modulus,  
    \textit{Right}: Phase). Zeros with positive charge are marked by $\times$ and zeros with negative charge by $\circ$.}
    \label{fig:sigplusnoise}
\end{figure*}

\section{Smoothed analysis at lower computational resolution}
\subsection{The PhaseJumps-coarse algorithm}
Consider again the problem of computing the zero set $\{F=0\} \cap \Omega_L$ of a possibly non-analytic function $F$, and let us introduce the following \emph{coarse sub-grid} of $\Lambda$:
\begin{align}\label{eq_cgrid}
\cgrid = \left\{ \cdelta k+i \cdelta j: k,j \in \mathbb{Z} \right\},  \mbox{ with } \cdelta := 2 
\left\lceil {\delta^{-1/2}}/{4} \right\rceil \delta,
\end{align}
which has spacing $\cdelta \approx \sqrt{\delta}$. The \calgo\, algorithm chooses some points of the finite grid
\begin{align*}
\cgrid_L := \cgrid \cap \Omega_L =
\left\{ \cdelta k+i \cdelta j: k,j \in \mathbb{Z}, |\cdelta k|, |\cdelta j| \leq L \right\}
\end{align*}
as numerical approximations of zeros, by applying three consecutive selection tests, formulated in terms of the twisted shifts \eqref{eq_ts}.

\noindent {\bf Step 1}. We compare oscillations relative to the \emph{fine scale} $\delta$ and the \emph{coarse scale}
\begin{align}\label{eq_def_dss}
\ccdelta := \left\lceil{\delta^{-1/2}}\right\rceil  \delta.
\end{align}

A point $\lambda \in \cgrid_L$ is selected if
\begin{align*}
|F_\lambda(\mu)| \geq 2 \lvert F_\lambda(\mu) - F_\lambda(\mu') \rvert,
\end{align*}
whenever $\mu, \mu' \in \Lambda$ are such that
\begin{align}\label{eq_intro_new}
|\mu|_\infty = |\mu'|_\infty =\ccdelta,
\mbox{ and } |\mu'-\mu| = \delta.
\end{align}
Here, $F_\lambda$ is the twisted shift, cf. \eqref{eq_ts}, and 
\eqref{eq_intro_new} describes all pairs of consecutive fine grid points on the boundary of the box with half side-length $\ccdelta$. This test is a variant of the small phase-oscillation condition \eqref{eq_intro_theta_2} that is better suited to the two-scale setting. Intuitively, the test is satisfied whenever $F$ vanishes near $\lambda$, because then the points $\mu$ at the boundary of the box of half side-length 
$\ccdelta \approx \sqrt{\delta}$ satisfy $|F_\lambda(\mu)| \approx |\nabla F_\lambda(\mu)| \sqrt{\delta}$, while the oscillations at the fine scale are $\lvert F_\lambda(\mu) - F_\lambda(\mu') \rvert \approx |\nabla F_\lambda(\mu)| \delta$, and the margin between the coarse and fine scale is wide enough to allow us to neglect second order errors. See Fig.~\ref{fig:grid_step1}.
\begin{figure*}[tbp]
    \centering
    \begin{tikzpicture}
    \draw (0, 0) node[inner sep=0] {\includegraphics[width=0.5\textwidth]{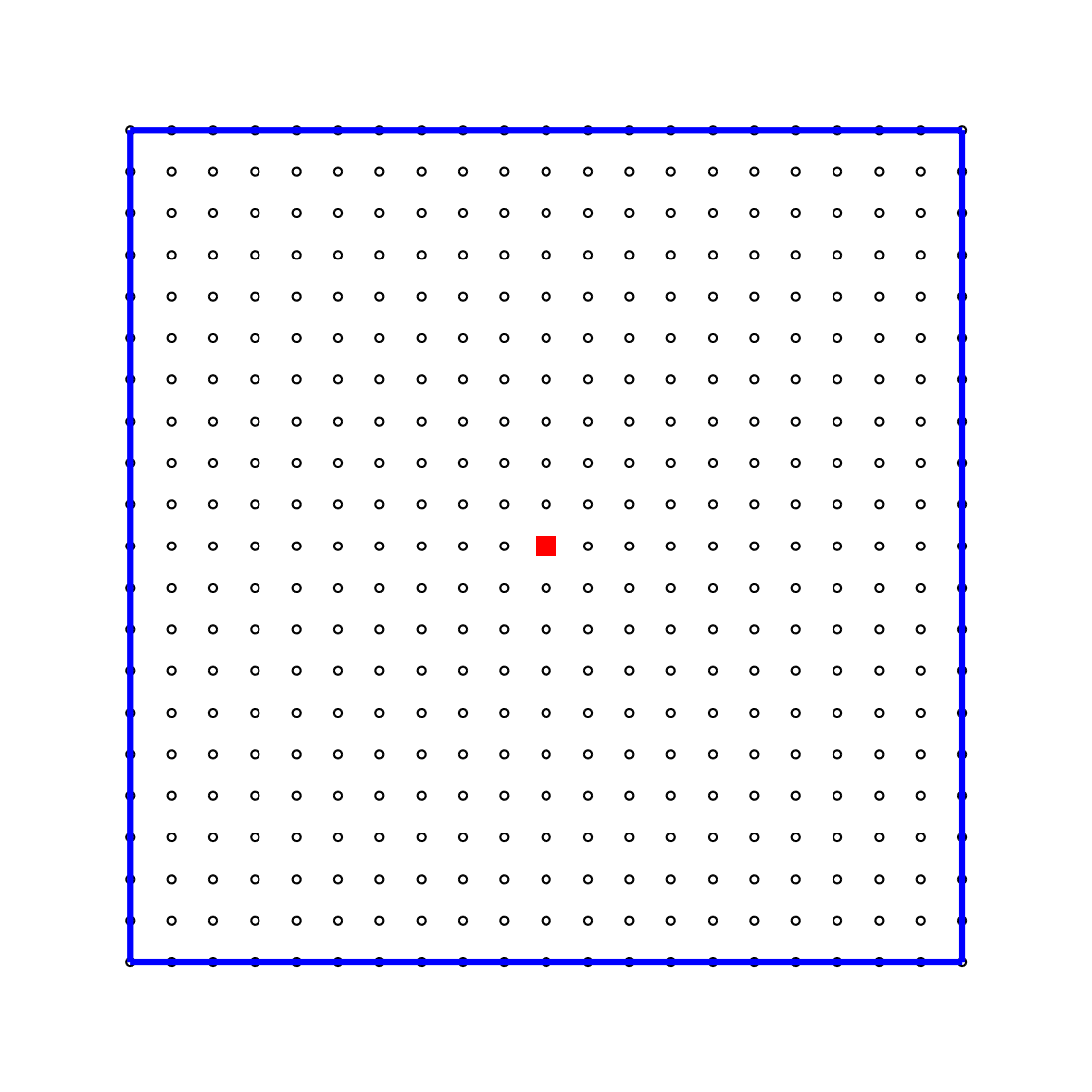}};
    \draw[thick,decorate,decoration={brace,amplitude=4pt}] (3.5, -3) -- (3.5, -3.37) node[midway, right,xshift=2pt,]{$\delta$};
    \draw[thick,decorate,decoration={brace,amplitude=4pt}] (4, 0) -- (4, -3.37) node[midway, right,xshift=2pt,]{$\delta^{**}$};
    \end{tikzpicture}
    \caption{Illustration of Step 1 of PhaseJumps-coarse (with $\delta=0.01$ and $\delta^{**}=0.1$). The center ($\lambda$, red square) of the coarse-scale test box (blue square) is selected if the oscillation of $F$ among all fine-scale neighboring points in the boundary of the box ($\mu, \mu'$, black circles) is small enough.}
    \label{fig:grid_step1}
\end{figure*}

\noindent {\bf Step 2}. In the same spirit as \algo, we then try to decide which boxes of half side-length $\ccdelta$ centered at the previously selected points $\lambda$ contain zeros of $F$ by computing an approximation of the argument change along the boundary of the box. We enumerate all points $\mu_j$ of the fine grid $\Lambda$ with $|\mu_j|_\infty=\ccdelta$ anti-clockwise and keep $\lambda$ only if
\begin{align}
\sum_{j=1}^N \arg \big[\,F_\lambda(\mu_{j}) \overline{F_\lambda(\mu_{j-1})}\,\big] \not=0.
\end{align}

\noindent {\bf Step 3}. Finally, we sieve the obtained set to enforce a minimal separation of $5 \ccdelta$ between different points. 

The algorithm is formally specified as follows.

\noindent\rule{\textwidth}{1pt}
\noindent {\bf Algorithm \calgo: \,}\\{Compute the (charged) zero set of $F$ inside the target domain $\domain = [-L,L]^2$.}\vspace{-0.2cm}\\
\noindent\rule{\textwidth}{1pt}

\alstep{Input}: A domain length $L\geq 1$, a grid spacing parameter $\delta >0$ and samples of a function $F\colon \mathbb{C} \to \mathbb{C}$ on the grid points $\gridlplus = \grid \cap [-(L+\ccdelta),L+\ccdelta]^2$, where $\ccdelta := \left\lceil{\delta^{-1/2}}\right\rceil  \delta$.

\alstep{Step 1}:  A point $\lambda \in \cgrid_L$ is \emph{selected} if
\begin{align}\label{st11}
|F_\lambda(\mu)| \geq 2 \lvert F_\lambda(\mu) - F_\lambda(\mu') \rvert,
\end{align}
whenever $\mu, \mu' \in \Lambda$ are such that
\begin{align}\label{st12}
|\mu|_\infty = |\mu'|_\infty =\ccdelta,
\mbox{ and } |\mu'-\mu| = \delta.
\end{align}
Let $\outseta$ be the set of all selected points.

\alstep{Step 2}. For each $\lambda \in \outseta$, consider all $\mu \in \Lambda$ such that $|\mu|_\infty = \ccdelta$, and order them as $\mu_1, \ldots, \mu_N$ so that $\arg(\mu_1) \leq \ldots \leq \arg(\mu_N)$.
Set $\mu_0 := \mu_N$. 

If $F_\lambda(\mu_{j})=0$ for some $j=1,\ldots,N$, set $\theta_\lambda :=0$. Otherwise, define 
\begin{align}\label{eq_tl}
	\theta_\lambda := \frac{1}{2\pi}\sum_{j=1}^N \arg \big[\,F_\lambda(\mu_{j}) \overline{F_\lambda(\mu_{j-1})}\,\big].
\end{align}
Let $\outsetb := \{\lambda \in \outseta: \theta_\lambda \not=0\}$.

\alstep{Step 3}. Use an off-the-shelf clustering algorithm to select a subset $\outset \subset \outsetb$ that is $5 \ccdelta$ separated:
\begin{align}\label{eq_5sep}
\inf \big\{ |\lambda-\lambda'|_\infty: \lambda, \lambda' \in \outset, \lambda\not=\lambda'\big\} \geq 5 \ccdelta
\end{align}
and is maximal with respect to that property, i.e., no proper superset satisfies \eqref{eq_5sep}.

\alstep{Output}: The set $\outset$ and the numbers $\{\theta_\lambda: \lambda \in \outset\}$.

\noindent(Note that the tests in Steps 1 and 2 only involve grid points 
$\lambda+\mu, \lambda+\mu',\lambda+\mu_j \in \gridlplus$.)
\\
\noindent\rule{\textwidth}{1pt}

We stress that \calgo\, is not introduced
as a competitive algorithm, but only as a means to approximately analyze the performance of \talgo, which is the preferred option in practice.

\subsection{Smoothed analysis}
We now analyze \calgo\, under a stochastic input model that is designed to describe its performance in practice. As input, we consider a random function $F\colon \mathbb{C} \to \mathbb{C}$ on the complex plane of the form:
\begin{align}\label{eq_F}
F= F^1+\sigma F^0,
\end{align}
where $F^1\colon \mathbb{C} \to \mathbb{C}$ is deterministic, $F^0$ is a zero mean random function
with distribution to be specified, and $\sigma>0$ is the \emph{noise level}. 

We will derive probabilistic guarantees that are uniform over deterministic functions $F^1$ satisfying suitable boundedness assumptions and depend explicitly on the noise level $\sigma$ and other relevant parameters such as the grid spacing $\delta$. This form of analysis, known as \emph{smoothed analysis} \cite{smooth}, lies between worst case and average case analysis, and to a great extent inherits the benefits of both. While average case analysis can show that an algorithm succeeds for generic or zero mean random inputs, these rarely model inputs found in practice. In contrast, smoothed analysis measures the performance of the algorithm under slight perturbations of possibly worst case inputs, and can show that \emph{theoretically bad inputs are fragile}, in the sense that a moderate amount of fluctuation (as often encountered in practice) can restore good performance. In applications to the short-time Fourier transform (see Section \ref{sec_stft}) smoothed analysis is closely related to the so-called \emph{noise assisted methods}, see, e.g. \cite{wu2009ensemble, MR3042191, colominas2014improved, torres2011complete}.

\subsection{Input model}\label{sec_input_model}
We consider the random function \eqref{eq_F} and specify $F^0$ and $F^1$. The assumptions we make are motivated by the application introduced in Section \ref{sec_stft}, so that, under the change of variables \eqref{eq_FF}, $F$ corresponds to the short-time Fourier transform
$V_g f$ of a random distribution\footnote{Here, we assume that the window function $g$ is Schwartz and we interpret \eqref{eq_stft} distributionally.} on $\mathbb{R}$:
\begin{align}\label{eq_f_signal}
f = f^1 + \sigma f^0,
\end{align}
where $f^0$ is standard (complex, Gaussian) white noise on $\mathbb{R}$ \cite{Hida1980}, $\sigma>0$ is the noise level, and $f^1$ is an arbitrary distribution with bounded short-time Fourier transform\footnote{In the jargon of time-frequency analysis, we say that $f^1$ belongs to the \emph{modulation space} $M^\infty(\mathbb{R})$ \cite{benyimodulation}.} -- a condition satisfied by any $L^2$ function, and also by all distributions commonly used in signal processing (Dirac deltas, periodizations and Fourier transforms thereof, etc.) \cite{benyimodulation}. 

We now spell out concrete assumptions. To quantify the smoothness and decay of $F^1$ we use the \emph{twisted derivatives}:
\begin{align}\label{eq_td}
	\Dop_1 F(z) = \partial F(z) - \tfrac{\bar z}{2} F(z), \quad \Dop_2 F(z) = \bar \partial F(z) + \tfrac{z}{2} F(z),
\end{align}
where $\partial F(x+iy) = \frac{1}{2} \big(\partial_x F(x+iy) - i \partial_y F(x+iy)\big)$ and $\bar\partial F(x+iy) = \frac{1}{2} \big(\partial_x F(x+iy) + i \partial_y F(x+iy)\big)$
are the usual complex derivatives. The merit of the twisted derivatives is that they commute with the twisted shifts \eqref{eq_ts}.

We assume that $F^1$ is twice continuously differentiable (in the real sense) and, moreover,
\begin{align}\label{A1}
 A:=\max\big\{\|F^1\|_\infty, \|\Dop_j F^1\|_\infty,
 \|\Dop_j \Dop_k F^1\|_\infty: j,k=1,2\big\} < \infty.
\end{align}
The random function $F^0$ is assumed to be 
a circularly symmetric complex Gaussian random field on $\C$ with covariance kernel 
\begin{align}\label{eq_H}
\E \Big[F^0(z) \cdot \overline{F^0(w)} \Big]=
H(z-w) e^{i \Im(z \bar w)}, \qquad z,w\in\mathbb{C},
\end{align}
where $H: \mathbb{C} \to \mathbb{C}$ is a function called \emph{twisted kernel}. Concretely this means that for every $z_1, \ldots, z_n \in \mathbb{C}$, the complex random vector $(F^0(z_1), \ldots, F^0(z_n))$ is normally distributed, circularly symmetric,
with mean zero and covariance matrix
$\big(H(z_k-z_j) e^{i \Im(z_k \bar z_j)}\big)_{k,j}$.

The particular structure of the covariance kernel is a form of stationarity introduced in \cite{hkr22} under the name of \emph{twisted stationarity}, and further studied in \cite{fhkr}. It means that the stochastics of $F^0$ are invariant under twisted shifts \eqref{eq_ts}: \[\E \Big[F^0(z) \cdot \overline{F^0(w)} \Big]=\E \Big[F_\xi^0(z) \cdot \overline{F_\xi^0(w)} \Big], \qquad \xi \in \mathbb{C}.\]

Of the twisted covariance kernel, we assume that 
\begin{align}\label{A2}
H
\in C^6 \mbox{ in the real sense}
\end{align}
and impose the normalization
\begin{align}\label{A3}
	H(0)&=1,
\end{align}
which means that $\var[F^0(z)]=1$. Consequently $\sigma^2$ is the noise variance in \eqref{eq_F}.
We also assume that
\begin{align}\label{A4}
	|H(z)| <1, \qquad z \in \C \setminus \{0\},
\end{align}
which means that $F^0(z)$, $F^0(w)$ are not deterministically correlated if $z\not=w$ --- because the determinant of their joint covariance matrix is
$1-|H(z-w)|^2$. We also assume that $F$ is almost surely $C^2$. Since $H$ is smooth, this is automatically the case, as long as a certain mild technical condition known as \emph{separability} is satisfied; see \cite[Section 1.1]{adler} \cite[Chapter 1]{level}.

As we discuss in Section \ref{sec_stft_input} and Lemma \ref{lem_stft}, for $f$ as in \eqref{eq_f_signal} and
\begin{align}\label{eq_f_stft}
	F(x+iy) := e^{-i xy}  \cdot V_g \, f \big(x/\sqrt{\pi},-y/\sqrt{\pi}\big),
	\end{align}
all assumptions are satisfied with
$H(x+iy) = e^{-i xy} \cdot	V_g g \big(x/\sqrt{\pi},-y/\sqrt{\pi} \big)$. In this case, the constant \eqref{A1} is $A \asymp \|V_g f^1\|_\infty$, so that $A/\sigma$ can be interpreted as a signal-to-noise ratio. There are also other interesting instances of the input model, such as \emph{random poly-entire functions} \cite{hahe13, hkr22}, which are important in mathematical physics \cite{MR2593994, vas00}.

\begin{rem}
As we show in Lemma \ref{lemma_nd}, under this input model, the zero set of $F$ is almost surely a discrete set (that is, each zero is isolated).
\end{rem}

\subsection{Performance guarantees}
Our main theoretical result gives guarantees for the computation of the zero set of a function $F$ restricted to the square \eqref{eq_Omega} from grid samples, providing an estimate for the Wasserstein distance between the atomic measures supported on the true zero set $\{F=0\} \cap \domain$ and on the computed set $\outset$ (up to a small boundary effect). For simplicity we assume that the corners of the computation domain lie on the coarse grid \eqref{eq_cgrid}. 

\begin{theorem}\label{th_main}
Fix a domain width $L \geq 1$, a noise level $\sigma>0$ and a grid resolution parameter $\delta>0$ such that $L/\cdelta \in \mathbb{N}$. Let a realization of a random function $F$ as in \eqref{eq_F}, with \eqref{A1}, \eqref{eq_H}, \eqref{A2}, \eqref{A3}, \eqref{A4}, be observed on the grid $\gridlplus$ and let $\outset$ be the output of the \calgo\, algorithm. Then there exist constants $C,c>0$ such that with probability at least
\begin{align}\label{eq_fp}
1-C \cdot L^2 \cdot \max\big\{1,\big(\log \tfrac{1}{\delta}\big)^2\big\} \cdot \exp\big(c{A^2}/{\sigma^2}\big) \cdot \delta
\end{align}
there is an injective map $\map\colon \{F=0\} \cap \domain \to \outset$ with the following properties:

$\bullet$ \emph{(Each zero is mapped into a near-by numerical zero)}
\begin{align}\label{eq_dist}
	|\map(\zeta) - \zeta|_ \infty \leq 2 \sqrt{\delta}, \qquad \zeta \in \{F=0\} \cap \domain.
\end{align}

$\bullet$ \emph{(Each numerical zero that is away from the boundary arises in this way)}

For each $\lambda \in \outset \cap \domainminus$ there exists
$\zeta \in \{F=0\} \cap \domain$ such that $\lambda=\map(\zeta)$.

$\bullet$ \emph{(Winding numbers are accurately computed)}
\begin{align}\label{wellcomp}
\charge_\zeta = \theta_{\map(\zeta)}, \qquad \zeta \in \{F=0\} \cap \domain.
\end{align}
(Here, the constants $C,c>0$ depend on the twisted kernel $H$, cf. \eqref{eq_H}.)
\end{theorem}

\subsection{Technical overview and comparison with other work}

Our work introduces and analyzes an algorithm for learning the zero set of the short-time Fourier transform (STFT). We model signals as continuous-time (analog) and assume they are observed through finitely many measurements. A central challenge is to clarify how the analog model relates to the resulting finite-dimensional measurement setting. In this sense, our work is related to similar endeavors on other nonlinear inverse problems for the STFT, such as phase retrieval \cite{MR5010357,MR4709358}.

To our knowledge, the Minimal Grid Neighbors (MGN) algorithm described in Section \ref{sec_1} is the most effective method to compute zeros of \emph{weighted analytic functions} from grid values. In our experience, other methods such as thresholding (using small values as proxies for zeros) or extrapolation (using the grid samples to approximately evaluate the function at an arbitrary point, so as to, say, apply an iterative root finder) lack the effectiveness or simplicity of MGN (see e.g., \cite[Sections 1.2 and 5]{efkr24}) and are not preferred in practice. Notwithstanding, certain proposed homotopy algorithms \cite{10163803} use MGN as an initialization step. While we are not aware of performance guarantees for MGN, a slight variation of MGN --- called Adapted Minimal Grid Neighbors (AMN) --- was introduced and analyzed in \cite{efkr24}. We developed \algo\, to address the 
problem of computing zeros of possibly \emph{non-analytic functions} as cost-effectively as with MGN/AMN. 

Let us now compare the performance guarantees that we prove for \calgo\, to those available for AMN. The following compares Theorem \ref{th_main} to 
\cite[Theorem 2.2]{efkr24}, which provides an estimate of the Wasserstein distance between the atomic measures supported on the true zero set of an analytic function and the set computed by AMN.
\begin{itemize}[leftmargin=0.3cm, itemsep=0.2cm,topsep=4pt]
\item[$\bullet$] The input model considered in \cite{efkr24} is a particular instance of the one introduced in Section \ref{sec_input_model}: there, functions are assumed to be analytic, while here functions may be non-analytic --- see also Lemma \ref{lem_stft}.

\item[$\bullet$] The failure probability in Theorem \ref{th_main} is $O(\log^2({1}/{\delta}) \cdot \delta)$, while the corresponding statement in \cite[Theorem 2.2]{efkr24} fails with probability $O(\log^2({1}/{\delta}) \cdot \delta^4)$.

\item[$\bullet$] The maximal distortion in Theorem \ref{th_main} is $O({\delta}^{1/2})$, while that in \cite[Theorem 2.2]{efkr24} is 
$O(\delta)$.
\end{itemize}
We believe that the weaker performance proved for \calgo\, in comparison to AMN is not an artifact of our argument, but rather that it is commensurate with the fundamentally more challenging task of finding zeros of non-analytic inputs. Indeed, the zero set of a non-zero analytic function $F$ is discrete, and, moreover, becomes well-separated after a moderate amount of noise is added to $F$ \cite{efkr24}. In contrast, zeros of non-analytic functions can lie along curves, which persists as \emph{almost nodal curves} after adding randomness; see Fig.~\ref{fig:herm_compare} for an illustration in the context of the short-time Fourier transform.
\begin{figure*}[tbp]
    \centering
    \begin{subfigure}[b]{0.47\textwidth}
    \includegraphics[width=0.98\textwidth]{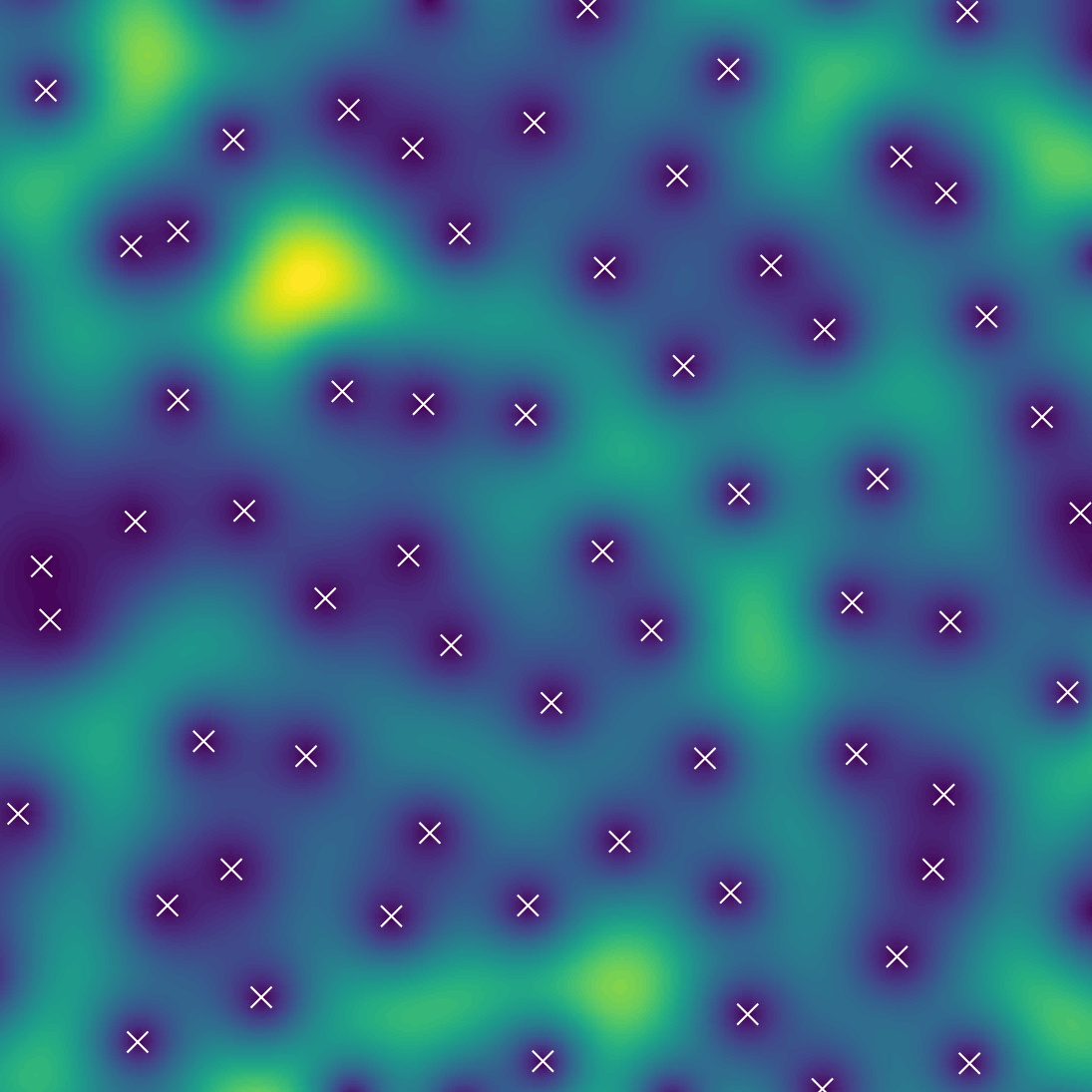}
    \end{subfigure}
    \hfill
    \begin{subfigure}[b]{0.47\textwidth}
    \centering
    \includegraphics[width=0.98\textwidth]{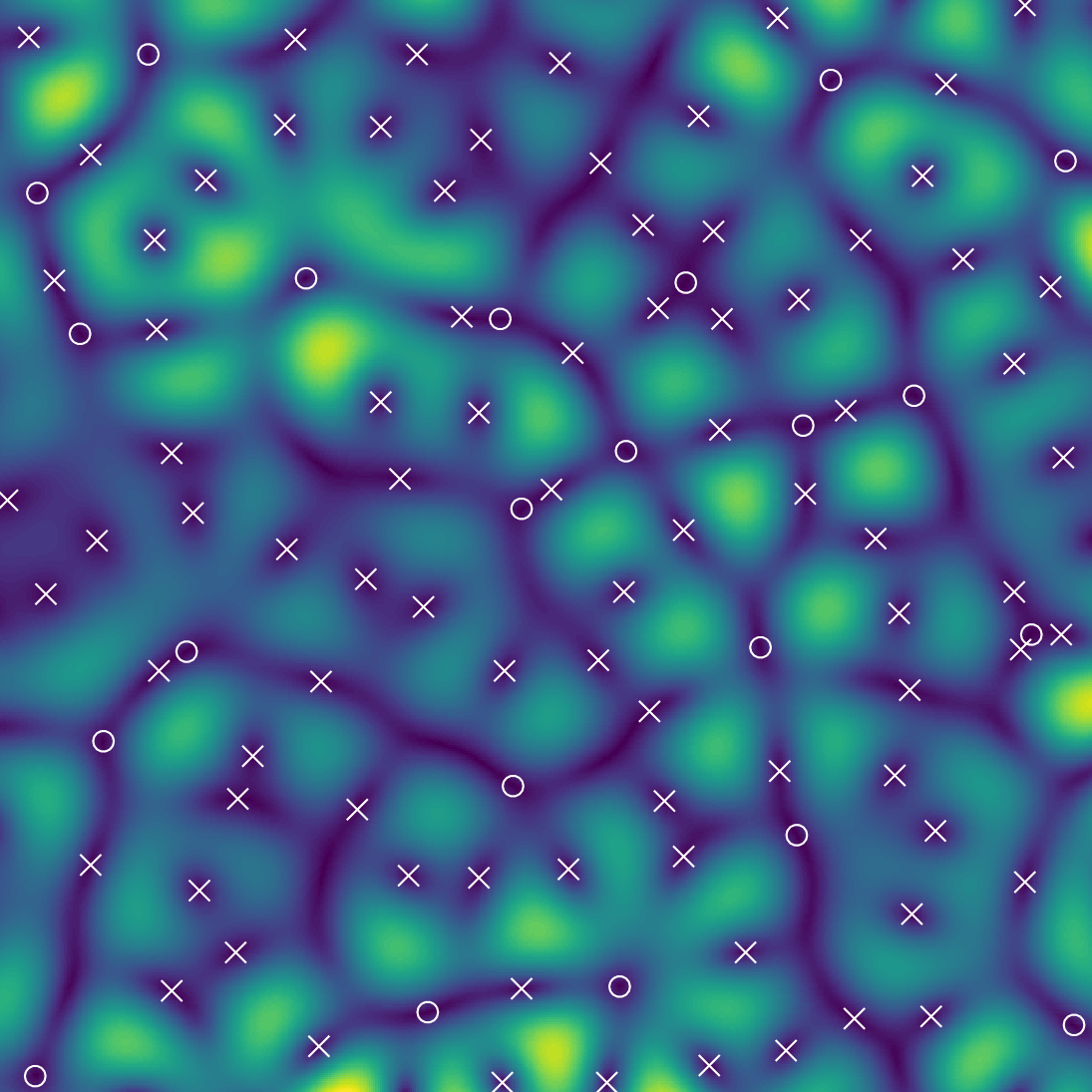}
    \end{subfigure}
    \caption{The squared absolute value of the STFT of (the same realization of) complex white noise using a Gaussian (left) or Hermite function of order 1 (right) as window, together with their zeros of positive (crosses) and negative (circles) charge. On the left-hand side, the magnitude increases locally at the same rate in all directions; on the right-hand side, zeros lie along ``low-magnitude curves'' (approximate nodal curves), which explains the difficulty in their computation. The positively (resp. negatively) charged zeros on the right-hand side correspond exactly to the saddle-points (resp. local maxima) of the function on the left-hand side.}
    \label{fig:herm_compare}
\end{figure*}
Thus, in comparison to the analytic case, much more randomness is required to regularize the zero set and turn it into a well-spread point pattern \cite{MR2966361,hor25}. In fact, as we show in Section \ref{sec_analysis}, analytic inputs form a ``thin'' set within our input model, and substantial effort is put into treating it separately, as a degenerate case (cf.\ Propositions \ref{prop_singular} and \ref{prop_gridevent}).

In the formulation of Theorem \ref{th_main} and the organization of its proof we tried to parallel those of \cite{efkr24}, so as to
better highlight the key differences between the analytic and non-analytic scenarios. Notably, as non-analytic functions depend non-trivially on the variables $z$ and $\bar{z}$, linearization arguments are much more subtle
because of different growth rates in different directions; see Fig.~\ref{fig:herm_compare}. Another central technical challenge in the present setting is that zeros are detected by means of topological index formulae, rather than by subharmonicity considerations, which requires new arguments. While the proofs in \cite{efkr24} rely on the theory of Gaussian Entire Functions \cite{NSwhat, gafbook}, in this article we use and further develop the parallel theory of Gaussian Weyl-Heisenberg Functions \cite{hkr22, fhkr}.

\subsection{Organization}
The remainder of the article is organized as follows. Section \ref{sec_analysis} analyzes the stochastic input model and the geometry of the corresponding zero sets and winding numbers. Section \ref{sec_proof_main} provides a proof of Theorem \ref{th_main}. In Section \ref{sec_num}, we provide several numerical experiments that illustrate the effectiveness of \algo\, in time-frequency analysis and the need for twisted shifts. We also conduct an empirical comparison with \calgo, which shows that \algo\, should be preferred in practice.

\section{Analysis of the Input Model}\label{sec_analysis}
\subsection{Preliminaries and notation}
Closed cubes are denoted $Q_r(w) := \{z \in \C: |z-w|_\infty \leq r\}$, while open disks are denoted
$D_r = \{z \in \mathbb{C}: |z|<r\}$. The \emph{twisted shifts} are defined by \eqref{eq_ts}. We let $\arg\colon \C\setminus\{0\} \to (-\pi,\pi]$ denote the principal branch of the argument, and we further set $\arg(0) := 0$ for convenience.

The complex derivatives (Wirtinger operators) of a function $F\colon \mathbb{C} \to \mathbb{C}$ are:
\begin{align*}
\partial F(x+iy) = \frac{1}{2} \big(\partial_x F(x+iy) - i \partial_y F(x+iy)\big),
\\
\bar\partial F(x+iy) = \frac{1}{2} \big(\partial_x F(x+iy) + i \partial_y F(x+iy)\big),
\end{align*}
while derivatives with respect to real variables $x,y$ are denoted $\partial_x F, \partial_y F$ and extended to multi-indices in the usual way:
\begin{align}
\partial^\alpha := \partial_x^{\alpha_1} \partial_y^{\alpha_2}, \qquad \alpha=(\alpha_1,\alpha_2) \in \mathbb{N}_0^2.
\end{align}
Directional derivatives are denoted:
\[\partial_{v}F(z)= v\partial F(z) + \bar v \bar \partial F(z)= v_1 \partial_x F(z) + v_2 \partial_y F(z),\]
where $v=v_1+iv_2 \in \C$ and $|v|=1$.

Letting $DF(z) \in \mathbb{R}^{2\times 2}$ be the differential matrix of $F\colon \mathbb{R}^2 \to \mathbb{R}^2$, we have the following expressions:
\begin{align}\label{eq_jacv}
\det DF(z) = |\partial F(z)|^2 - |\bar\partial F(z)|^2= \Im\big[\partial_v F(z) \cdot \overline{\partial_{-iv}F(z)} \big], \qquad |v|=1.
\end{align}
We also use the \emph{twisted derivatives} \eqref{eq_td}, which commute with the twisted shifts \eqref{eq_ts}:
\begin{align}\label{eq_comm}
\Dop_j (F_w) = (\Dop_j F)_w, \qquad w \in \mathbb{C}.
\end{align}

The \emph{charge} of $F$ at a point $z \in \mathbb{C}$ is
\begin{align*}
\charge_z = \sgn \det DF (z) \in \{-1,0,1\}.
\end{align*}
For a curve $\gamma\colon [a,b] \to \mathbb{C}$, $\gamma(t) = x(t) + i y(t)$, the integrals of the forms $F dz$ and $F d\bar{z}$ are defined as usual:
\begin{align*}
\int_{\gamma} F(z)\,dz = \int_{a}^b F(\gamma(t)) (x'(t)+iy'(t))\,dt, \\
\int_{\gamma} F(z)\,d\bar{z} = \int_{a}^b F(\gamma(t)) (x'(t)-iy'(t))\,dt.
\end{align*}
We also use the following notation for the \emph{total derivative}:
\begin{align*}
dF = \partial F dz + \bar{\partial} F d\bar{z},
\end{align*}
and denote by $dA$ the differential of the (Lebesgue) area measure on $\mathbb{C}$.

Throughout the manuscript we consider the \emph{acquisition lattice} \eqref{eq_Lambda} with spacing $\delta>0$ and the \emph{computation domain} \eqref{eq_Omega} with half side-length $L>0$.
Associated with the grid-spacing $\delta$, we consider the related quantities
\begin{align*}
\cdelta = 2 
\left\lceil {\delta^{-1/2}}/{4} \right\rceil \delta,
\qquad
\ccdelta = \left\lceil{\delta^{-1/2}}\right\rceil  \delta,
\end{align*}
which are $\approx \sqrt{\delta}$, and the \emph{coarse grid} \eqref{eq_cgrid}, with spacing $\cdelta$. The restriction of grids to a square are denoted
\begin{align}
\Lambda_M = \{ \lambda \in \Lambda: |\lambda|_\infty \leq M\},
\qquad\qquad
\cgrid_M = \{ \lambda \in \cgrid: |\lambda|_\infty \leq M\}.
\end{align}
In particular, we write $\Lambda_L, \Lambda_{L + \delta}$, etc. for grid points lying on or close to the computation domain \eqref{eq_Omega}.

\subsection{Twisted shifts and Euclidean derivatives}
As mentioned, the twisted derivatives \eqref{eq_td}, commute with the twisted shifts \eqref{eq_ts}, cf.\ \eqref{eq_comm}. We now observe that for small parameters, Euclidean derivatives and twisted shifts enjoy an approximately Euclidean covariance relation.
\begin{lemma}\label{lem_dervery}
The following holds for a (in the real sense) differentiable function $F\colon\mathbb{C} \to \mathbb{C}$:
\begin{align}
|\partial_x F_{w+w'}(z)| \leq |w'|_\infty |F_w(z+w')| + |\partial_x F_w(z+w')|, \qquad z,w,w' \in \mathbb{C},
\\
|\partial_y F_{w+w'}(z)| \leq |w'|_\infty |F_w(z+w')| + |\partial_y F_w(z+w')|, \qquad z,w,w' \in \mathbb{C}.
\end{align}
\end{lemma}
\begin{proof}
We note that $F_{w+w'}(z)=e^{i\Im(w'\overline{w})}
e^{-i \Im(z \overline{w'})} F_w(z+w')$ and compute
\begin{align*}
e^{-i\Im(w'\overline{w})}\partial_x F_{w+w'}(z) = i \Im(w') e^{-i \Im(z \overline{w'})} F_w(z+w') + e^{-i \Im(z \overline{w'})} \partial_x F_w(z+w'),
\\
e^{-i\Im(w'\overline{w})}\partial_y F_{w+w'}(z) = -i \Re(w') e^{-i \Im(z \overline{w'})} F_w(z+w') + e^{-i \Im(z \overline{w'})} \partial_y F_w(z+w'),
\end{align*}
from where the estimate follows.
\end{proof}
\subsection{The sparse grid}
With a slight abuse of language, we define the \emph{sparse grid} as the following subset of $\Lambda_L$:
\begin{align}\label{eq_gs}
\gridsparse:= \big\{ \mu \in \Lambda_L: |\mu- \lambda|_\infty = \ccdelta\text{ for some } \lambda \in \cgrid \big\},
\end{align}
which consists of all (fine) grid points $\mu$ that lie on the boundary boxes $Q_{\ccdelta}(\lambda)$, $\lambda \in \cgrid$; see Figure \ref{fig:grids}.
This set is instrumental in the analysis of \calgo\, because the comparison in Step 1 involves evaluating $F$ on points $\lambda+\mu \in \gridsparseno_{L+\ccdelta}$, cf. \eqref{st11}. We record the following simple facts. 
\begin{lemma}\label{lemma_sparse}
Let $L\geq 1$, $0<\delta \leq 1$ and $\gridsparse$ be given by \eqref{eq_gs}. Then
the cardinality of the sparse grid satisfies
\begin{align}\label{eq_eg}
\# \gridsparse \leq C L^2\delta^{-3/2},
\end{align}
for a universal constant $C$.

In addition, assume that $\delta<1/16$. Then
given $\xi \in \Omega_L$ and $v\in\mathbb{C}$ with $|v|=1$, there exists $t \in [2,4]$ and $\tau \in \gridsparseplus$ such that $|\xi+\sqrt{\delta}t v - \tau | < \delta$. 
\end{lemma}
\begin{proof}
To prove \eqref{eq_eg}, note that if $\lambda \in \cgrid$ is such that there exists $\mu \in \Lambda_L$ with $|\mu- \lambda|_\infty = \ccdelta$, then $|\lambda|_\infty \leq L+2$, so there are at most $\lesssim L^2/\delta$ such points. For each such $\lambda$, there are at most $\lesssim \ccdelta/\delta \asymp \delta^{-1/2}$ points $\mu \in \Lambda$ with $|\mu- \lambda|_\infty = \ccdelta$. Hence, \eqref{eq_eg} follows from a union bound; see Figure \ref{fig:grids}. 

For the second assertion, consider the set of vertical and horizontal segments \[M = \{ z \in \mathbb{C}\,:\, |z- \lambda|_\infty = \ccdelta\text{ for some } \lambda \in \cgrid \}.\] The grid $\cgrid$ 
has spacing $\cdelta$ and the following shows that $\cdelta <  \ccdelta \leq 2 \cdelta$:
\begin{align*}
    \ccdelta
    & \geq \sqrt{\delta}
    > \big(\sqrt{\delta}/2+2\delta\big)
    = 2 \delta \big(\delta^{-1/2}/4+1\big)
    > \cdelta,
    \\
    \ccdelta
    & = \left\lceil{\delta^{-1/2}}\right\rceil  \delta
    \leq 4 \left\lceil{\delta^{-1/2}}/4\right\rceil  \delta
    = 2 \cdelta.
\end{align*}
For each $\lambda \in \cgrid$,
we consider the four line segments
forming the box
$\{z\,:\,|z- \lambda|_\infty = \ccdelta\}$. Since $\cdelta <  \ccdelta \leq 2 \cdelta$, the top or bottom horizontal segments 
in $M$ defined by horizontally neighboring $\lambda \in \cgrid$ are overlapping, while the corresponding vertical segments are at most $\cdelta$ apart (and vice versa; see Figure~\ref{fig:grids}). Hence,
$M$ is a mesh and the longest line segment disjoint of $M$ has length at most $\sqrt{2}\cdelta$.
Since
\begin{equation*}
    \big(4-2\big)\sqrt{\delta} 
    > \sqrt{2}\sqrt{\delta}
    > \sqrt{2}\cdelta,
\end{equation*}
the line segment $\{\xi+\sqrt{\delta}t v: t \in [2,4]\}$ must intersect $M$. That is, there exists $t \in [2,4]$ such that
$z := \xi+\sqrt{\delta}t v \in M$.
This means that there exists $\lambda \in \cgrid$ such that $|z- \lambda|_\infty= \ccdelta$. We note that 
$|z|_\infty=\lvert\xi+\sqrt{\delta}t v\rvert_\infty < L + 1/4 \cdot 4 = L+1$ and take $\tau \in \grid$ with $|z - \tau| < \delta$ and in the same box boundary, $|\tau-\lambda|_\infty= \ccdelta$, with the property that $|\tau|_\infty\leq|z|_\infty$. Hence $\tau \in \gridsparseplus$, as desired.
\end{proof}

\begin{figure*}[tbp]
    \centering
    \begin{tikzpicture}
    \draw (0, 0) node[inner sep=0] {\includegraphics[width=0.78\textwidth]{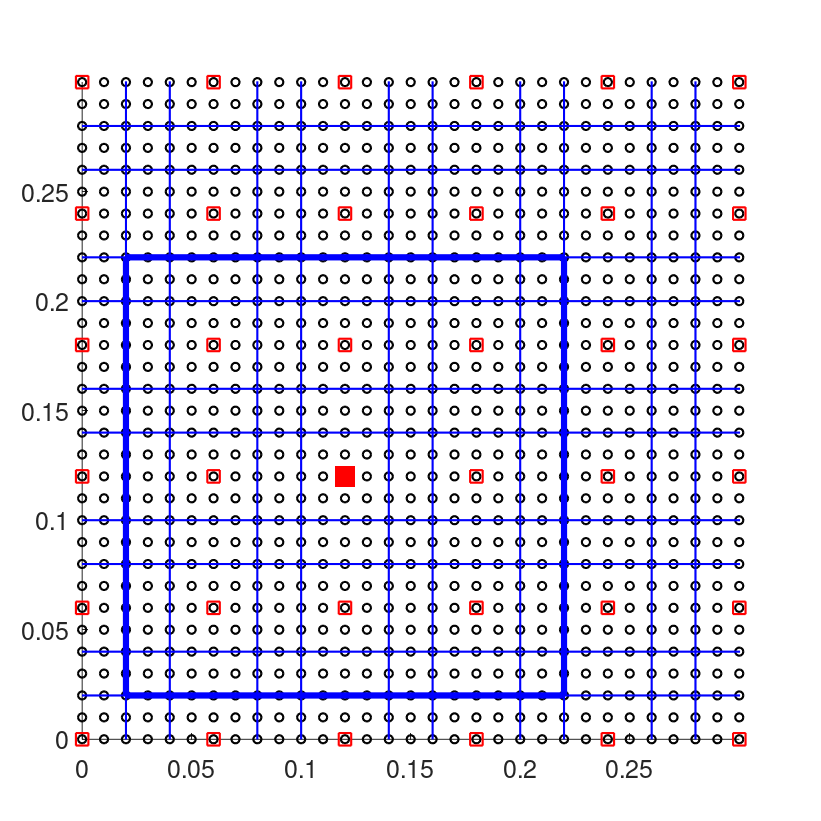}};
    \draw[thick,decorate,decoration={brace,amplitude=4pt}] (5.7, -5.13) -- (5.7, -5.52) node[midway, right,xshift=2pt,]{$\delta$};
    \draw[thick,decorate,decoration={brace,amplitude=4pt}] (6.2, -3.3) -- (6.2, -5.52) node[midway, right,xshift=2pt,]{$\delta^*$};
    \draw[thick,decorate,decoration={brace,amplitude=4pt}] (6.8, -1.82) -- (6.8, -5.52) node[midway, right,xshift=2pt,]{$\delta^{**}$};
    \end{tikzpicture}
    \caption{Illustration of the different grids. Black circles mark the elements of the fine grid $\Lambda_L$, red squares mark the elements of the coarse grid $\Delta$, blue lines mark the set $M$
used in the proof of Lemma \ref{lemma_sparse}. The thick blue lines illustrate the part of the set
    $M$ originating from the element $\lambda \in \Delta$ highlighted by the filled red square. The illustration is based on $\delta=1/100$, and thus $\delta^*=6/100$ and $\delta^{**}=10/100$. }
    \label{fig:grids}
\end{figure*}

\subsection{The short-time Fourier transform as a main example}\label{sec_stft_input}
We now discuss why the short-time Fourier transform of a signal impacted by noise is an instance of the input model. For the formulation, we use the Schwartz class $\mathcal{S}(\mathbb{R})$ and the so-called modulation space $M^\infty(\mathbb{R})$ consisting of distributions with bounded short-time Fourier transform \cite{benyimodulation}. We also interpret Gaussian white noise as a random distribution (generalized stochastic process) \cite{Hida1980}, which allows us to calculate its short-time Fourier transform; see \cite[Section 6.1]{hkr22}, or \cite{bh} for a different approach.
\begin{lemma}\label{lem_stft}
Let $g \in \mathcal{S}(\mathbb{R})$ be normalized by $\|g\|_2=1$, $f^0$  standard complex-valued Gaussian white noise, $f^1 \in M^\infty(\mathbb{R})$ and $\sigma>0$. Set $F=F^1+\sigma F^0$ with
	\begin{align}
	F^j(z) := e^{-i xy}  \cdot V_g \, f^j \big(x/\sqrt{\pi},-y/\sqrt{\pi} \big), \qquad z=x+iy.
	\end{align}
    Then $F$ is almost surely twice continuously differentiable in the real sense, and
\eqref{A1}, \eqref{eq_H}, \eqref{A2}, \eqref{A3}, \eqref{A4} hold with
	$H(x+iy) = e^{-i xy} \cdot
	V_g g \big(x/\sqrt{\pi},-y/\sqrt{\pi} \big)$. In addition, the constant in \eqref{A1} satisfies $A \asymp \|V_g f^1\|_\infty$.
\end{lemma}
\begin{proof}
This is proved in \cite[Lemma 6.1]{hkr22} for the zero mean case and in \cite[Lemma 9.2]{fhkr} in full generality. The references only provide bounds for the first order derivatives $\Dop_1 F^1, \Dop_2 F^1$ but the same argument applies to second order derivatives
$\Dop_j \Dop_k F^1$ and shows that
$A \asymp \|V_g f^1\|_\infty$.
\end{proof}

\subsection{Gaussian Weyl-Heisenberg functions}
For the remainder of the section, we let
$F= F^1+\sigma F^0$, with $\sigma>0$, be a random function satisfying the assumptions of Section \ref{sec_input_model}, namely: $F$ is almost surely twice continuously differentiable in the real sense, and
\eqref{A1}, \eqref{eq_H}, \eqref{A2}, \eqref{A3}, \eqref{A4} hold. This class of random functions was introduced in \cite{hkr22} and further studied in \cite{fhkr}. For short, we say that $F$ is a \emph{Gaussian Weyl-Heisenberg function} (GWHF) with \emph{twisted kernel} $H$, cf. \eqref{eq_H}.

We note the following symmetry of the twisted kernel $H$: setting $z=0$ or $w=0$ in \eqref{eq_H} we see that
\begin{align}\label{eq_Hsym}
H(-z)=\overline{H(z)}, \qquad z \in \mathbb{C}.
\end{align}
It is important to note the following uniformity property: for arbitrary $w \in \mathbb{C}$, the shifted function $F_w = F^1_w + \sigma F^0_w$ is another instance of the same stochastic model, with the same twisted kernel $H$, while the constant $A$ associated to $F^1_w$ is the same as that associated with $F^1$ in \eqref{A1}.

\subsection{Excursion bounds} As a first step, we look into the derivatives of the input function $F$.

\begin{lemma}\label{lemma_der}
For multi-indices $\alpha,\beta$ with $|\alpha|,|\beta|\leq 2$ we have\begin{align*}
\cov \big[ \partial^\alpha F^0(x + iy), \partial^\beta F^0(u+iv)\big]=
\partial^\alpha_{(x,y)} 
\partial^\beta_{(u,v)} \Big[
H(z-w) e^{i \Im(z \bar w)} \Big],\quad
z=x+iy,w=u+iv \in \mathbb{C}.
\end{align*}
\end{lemma}
\begin{proof}
Since $F$ is almost surely $C^2$, it is a separable Gaussian field. This, together with the fact that $H \in C^6$ allows one to commute expectation and derivatives up to order 2; see \cite[Chapter 1]{level} or \cite[Section 1.1]{adler}.
\end{proof}
Next, we investigate the typical size of the derivatives of the input function $F$.
\begin{lemma} \label{lem:excu}
There exist constants $c,c',C>0$ that depend on the twisted kernel $H$ such that for every multi-index $|\alpha| \leq 2$, the following excursion bound holds
\begin{align}\label{eq_le2}
  \mathbb{P} \bigg( \sup_{|z|\leq 10, |w| \leq L } |\partial^\alpha F_w(z)|>s \bigg) \leq C L^2 
  e^{c'\tfrac{A^2}{\sigma^2}}
  e^{-c\tfrac{s^2}{\sigma^2}}, \qquad s \geq 0, \qquad L \geq 1.
\end{align}
\end{lemma}
\begin{proof}
\noindent {\bf Step 1}. Fix a multi-index $\alpha$ with $|\alpha|\leq 2$. We shall apply the Borell–TIS inequality (see e.g., \cite[Theorem 2.9]{level} or \cite[Theorem 2.1.1]{adler}) to the real-valued Gaussian random fields
\[G_1(z):= \mathrm{Re}\big[\partial^\alpha F^0(z)\big], \quad G_2(z):= \mathrm{Im}\big[\partial^\alpha F^0(z)\big]\]
on the domain $\{z: |z| < 11 \}$.

Let us denote $K_j(\xi,\chi) = \mathbb{E} [ G_j(\xi) \overline{G_j(\chi)}]$. By Lemma \ref{lemma_der}, for $|\xi|,|\chi| \leq 11$,
\begin{align*}
\var[G_j(\xi) -  G_j (\chi)] = 
K_j(\xi,\xi) + K_j(\chi,\chi) - 2 \Re [K_j(\xi,\chi)] \lesssim \sup_{|\beta| \leq 5, |z| \leq 22} |\partial^\beta H(z)| \cdot |\xi-\chi|.
\end{align*}
This estimate allows us to apply Dudley's inequality because it compares the metric induced by $G_j$ to the usual Euclidean metric, and provides bounds for the corresponding entropy numbers \cite[Chapter 2]{level}, \cite[Chapter 2]{adler}. By Dudley's inequality \cite[Theorem 2.10]{level} \cite[Theorem 1.3.3]{adler}, we have the estimate
\begin{align*}
B_j &:= \E \big[\sup_{|z|\leq 11} G_j(z)\big] \leq C, \qquad j=1,2,
\end{align*}
where the constant $C$ depends only on $H$. Similarly, by Lemma \ref{lemma_der},
\begin{align*}
\sigma_{j}^2 &:= \sup_{|z|\leq 11}\var G_j(z) \leq C, \quad j=1,2,
\end{align*}
with the constant $C$ depending only on $H$. By the Borell-TIS inequality (see, e.g., \cite[Theorems 2.8 and 2.9]{level}), 
\begin{align} \label{eq:excub}
 \mathbb{P} \bigg( \sup_{|z|\leq 11}  G_j(z)>s \bigg) = 
  \mathbb{P} \bigg( \sup_{|z|\leq 11}  G_j(z)-B_j>s-B_j \bigg)
 \leq 2  e^{-\frac12(s - B_j)^2/\sigma_j^2}.
\end{align}
Here, we may assume that $\sigma_{j}$ is not zero, because otherwise $G_j(z)$ would be almost surely $0$
for all $|z|\leq 11$, in which case the excursion bound \eqref{eq:excub} holds trivially. Arguing similarly with $-G_j$ and collecting all estimates, we obtain \begin{align}\label{eq_le2a}
  \mathbb{P} \bigg( \sup_{|z|<11} |\partial^\alpha F^0(z)|>s \bigg) \leq C_1 e^{-c_1 s^2}, \qquad s \geq 0.
\end{align}
Here and for the remainder of the proof $C_j, c_j$ will denote constants that depend only on the twisted kernel $H$.

\smallskip

\noindent {\bf Step 2}. We now compare Euclidean and twisted derivatives. For an arbitrary $G: \mathbb{C} \to \mathbb{C}$ that is $C^2$ in the real sense we define
\begin{align*}
\|G\|_* := \sup_{|z|\leq 10}
\max\big\{|G(z)|, |\Dop_j G(z)|, |\Dop_j \Dop_i G(z)|: i,j=1,2\big\},
\\
\|G\|_{**} := \sup_{|z|\leq 11}
\max\big\{|G(z)|, |\Dop_j G(z)|, |\Dop_j \Dop_i G(z)|: i,j=1,2\big\}.
\end{align*}
Then
\begin{align}\label{eq_then}
\|G\|_* \asymp \sup_{|\alpha|\leq 2}\sup_{|z| \leq 10} |\partial^\alpha G(z)|,
\end{align}
where the implied constants are absolute. 
The advantage of the introduced expressions is that they are naturally compatible with twisted shifts. For example, by the commutation relation
\eqref{eq_comm},
\begin{align*}
\|(G_{w_1})_{w_2}\|_*&=
\sup_{|z|\leq 10}
\max_{i,j=1,2}\big\{|(G_{w_1})_{w_2}(z)|, |((\Dop_j G)_{w_1})_{w_2}(z)|, |((\Dop_j \Dop_i G)_{w_1})_{w_2}(z)|\big\}
\\
&=
\sup_{|z|\leq 10}
\max_{i,j=1,2}\big\{|G(z+w_1+w_2)|, |\Dop_j G(z+w_1+w_2)|, |\Dop_j \Dop_i G(z+w_1+w_2)|\big\}
\\
&=\|G_{w_1+w_2}\|_*.
\end{align*}
Similarly,
\begin{align}\label{eq_r}
\begin{aligned}
\sup_{|w|\leq 1} \|G_w\|_* =\sup_{|w|\leq 1}\sup_{|z|\leq 10}
\max_{i,j=1,2}\big\{|G(z+w)|, |\Dop_j G(z+w)|, |\Dop_j \Dop_i G(z+w)|\big\}=\|G\|_{**}.
\end{aligned}
\end{align}
As a consequence of \eqref{eq_then}, the excursion bounds \eqref{eq_le2a} can be summarized as
\begin{align}\label{eq_le2b}
  \mathbb{P} \big( \|F^0\|_{**}>s \big) \leq C_2 e^{-c_2 s^2}, \qquad s \geq 0,
\end{align}
while \eqref{A1} and \eqref{eq_comm} imply that $\|F^1\|_{**} \leq A$.
Thus, 
\begin{align*}
\|F\|_{**} \leq A + \sigma \|F^0\|_{**},
\end{align*}
and, therefore, for $s \geq A$,
\begin{align}\label{eq_Fa}
\mathbb{P} \big( \|F\|_{**} >s\big) 
\leq 
\mathbb{P} \big( \|F^0\|_{**} >(s-A)/\sigma\big) 
\leq C_2 e^{-c_2\tfrac{(s-A)^2}{\sigma^2}}.
\end{align}
\smallskip
\noindent {\bf Step 3}.
Finally, we cover the disk $\overline{D}_L$ with $O(L^2)$ translates of the unit disk, which we denote $\overline{D}_1+\omega^L_{j,k}$. We now make the crucial observation that the functions $F_{\omega^L_{j,k}}$ satisfy the same assumptions as $F$, with uniform estimates. Hence, we can apply \eqref{eq_Fa} to each $F_{\omega^L_{j,k}}$ in lieu of $F$, and with help from
\eqref{eq_then} and \eqref{eq_r} bound
\begin{align*}
&\mathbb{P} \bigg( \sup_{|z|\leq 10, |w| \leq L } |\partial^\alpha F_w(z)|>s \bigg)
\leq \sum_{j,k}
\mathbb{P} \bigg( \sup_{|z|\leq 10, |w| \leq 1} |\partial^\alpha F_{w+\omega^L_{j,k}}(z)|>s \bigg)
\\
&\quad\leq \sum_{j,k}
\mathbb{P} \bigg( \sup_{|w| \leq 1} \|F_{w+\omega^L_{j,k}}\|_* > c_3 s\bigg)
= \sum_{j,k}
\mathbb{P} \bigg( \sup_{|w| \leq 1} \|(F_{\omega^L_{j,k}})_w\|_* > c_3 s\bigg)
\\
&\quad= \sum_{j,k}
\mathbb{P} \Big( \| F_{\omega^L_{j,k}} \|_{**}>c_3 s \Big),
\end{align*}
for a constant $c_3$. We note that \eqref{eq_le2} holds trivially for $0 \leq s < A/c_3$, provided that we choose, as we may, $c'>c/c_3^2$. We assume that $c_3 s \geq A$ and use \eqref{eq_Fa} to get
\begin{align*}
&\mathbb{P} \bigg( \sup_{|z|\leq 10, |w| \leq L } |\partial^\alpha F_w(z)|>s \bigg)
\leq C_5 L^2 \exp \big[-c_2\tfrac{(c_3 s-A)^2}{\sigma^2} \big]
\leq C_5 L^2 \exp \big[-\frac{c_2}{2}\tfrac{c_3^2 s^2-2A^2}{\sigma^2} \big],
\end{align*}
where we used that
$c_3^2s^2=\big((c_3s-A)+A \big)^2 \leq2 \big((c_3s-A)^2 + A^2\big)$.
\end{proof}

\subsection{The argument principle}
We note the following sufficient condition for the accurate calculation of the variation of the complex argument along a curve.
\begin{lemma}[Unproblematic calculation of argument change]\label{lemma_arg}
Let $\tau \in \C \setminus \{0\}$ and $\alpha\colon [a,b] \to \C$ continuously differentiable. Suppose that $\Re\big[\overline{\tau} \alpha(t)\big]>0$, for all $t \in [a,b]$. Then \[\Im\,\int_\alpha \frac{dz}{z} = \arg\big[\alpha(b)\overline{\alpha(a)} \big].\]
\end{lemma}
\begin{proof}
We assume without loss of generality that $|\tau|=1$, let $\log$ be the principal branch of the logarithm, and compute
\begin{align*}
\Im\,\int_\alpha \frac{dz}{z} &=
\Im\,\int_{\overline{\tau} \alpha} \frac{dz}{z}
= \Im \big[\log(\overline{\tau} \alpha(b)) - 
\log(\overline{\tau} \alpha(a)) \big]
\\
&=\arg(\overline{\tau} \alpha(b)) - 
\arg(\overline{\tau} \alpha(a)).
\end{align*}
Since $\arg(\overline{\tau} \alpha(b)),\arg(\overline{\tau} \alpha(a)) \in (-\pi/2,\pi/2)$, we have $\arg(\overline{\tau} \alpha(b)) - 
\arg(\overline{\tau} \alpha(a)) \in (-\pi,\pi)$ and, therefore,
\begin{equation*}
\arg(\overline{\tau} \alpha(b)) - \arg(\overline{\tau} \alpha(a))
= \arg\big[\overline{\tau} \alpha(b) \overline{\overline{\tau} \alpha(a)}\big]=\arg\big[\alpha(b)\overline{\alpha(a)} \big].\qedhere
\end{equation*}
\end{proof}

\subsection{Singularity of the covariance matrix}\label{sec_scov}
In order to perform smoothed analysis of \calgo\, we will study correlations between $F$ and its derivatives at various points. As a first step, we consider the zero-mean random vector
\begin{align}\label{eq_vector}
\big(F^0(0), \partial F^0(0), \bar{\partial} F^0(0) \big).
\end{align}
The following lemma shows that the components of \eqref{eq_vector} may exhibit deterministic relations.
\begin{lemma}\label{le_gef}
Suppose that $H(z)=e^{-|z|^2/2}$, $z \in \mathbb{C}$, and define $G^0(z) = e^{|z|^2/2} F^0(z)$. Then $G^0$ is almost surely an entire function, and, as a consequence, almost surely,
\begin{align}\label{eq_cc1}
\bar{\partial} F^0(z) + \tfrac{z}{2} F^0(z) = 0, \qquad z \in \mathbb{C}.
\end{align}
In addition $F^0(0)$ and $\partial F^0(0)$ are independent standard complex normal variables. 
\end{lemma}
\begin{proof}
We first compute the covariance kernel of the random Gaussian function $G^0$:
\begin{align*}
\mathbb{E} \big[ \,G^0(z) \cdot \overline{G^0(w)}\, \big] &= e^{|z|^2/2} e^{|w|^2/2} H(z-w) e^{i \Im(z\bar{w})}
\\
& = e^{|z|^2/2} e^{|w|^2/2} e^{-|z-w|^2/2} e^{i \Im(z\bar{w})} = e^{z \bar{w}}.
\end{align*}
By Lemma \ref{lemma_der}, we can then interchange expectation and differentiation and compute the covariance kernel of $\bar{\partial} G^0$:
\begin{align*}
\mathbb{E} \big[ \,\bar{\partial} G^0(z) \cdot\overline{\bar{\partial} G^0(w)} \,\big] 
=\mathbb{E} \big[ \,\bar{\partial} G^0(z) \cdot \partial\overline{G^0(w)} \,\big] 
= \partial_w \bar\partial_z e^{z\bar{w}} = 0.
\end{align*}
Thus $\bar \partial G^0$ is almost surely the zero function, $G^0$ is almost surely entire, and 
\begin{align*}
\bar\partial F^0(z) = \bar\partial \big[ e^{-z \bar{z}/2} G^0(z) \big] = e^{-|z|^2/2} \left( \bar \partial G^0(z) - \tfrac{z}{2} G^0(z) \right) = - \tfrac{z}{2} F^0(z), \qquad z \in \mathbb{C},
\end{align*}
which proves \eqref{eq_cc1}. (Recall that we assume that $F$ is almost surely continuous.)

Similarly,
\begin{align*}
\mathbb{E} \big[ \,G^0(z) \cdot \overline{\partial G^0(w)}\, \big] &= \bar\partial_w e^{z \bar{w}} = z e^{z \bar{w}},
\end{align*}
and
\begin{align*}
\mathbb{E} \big[ \,\partial G^0(z) \cdot \overline{\partial G^0(w)}\, \big] &=
\partial_z
\bar\partial_w e^{z \bar{w}} = e^{z \bar{w}} + z\bar{w} e^{z \bar{w}}.
\end{align*}
This shows that the covariance matrix of the Gaussian vector $(F^0(0),\partial F^0(0))=(G^0(0),\partial G^0(0))$ is the identity, and gives the remaining claim.
\end{proof}

The function $G^0$ in Lemma \ref{le_gef} is known as a \emph{Gaussian entire function (GEF)} and can also be described as a random Taylor series with adequately normalized independent complex normal coefficients \cite{NSwhat, gafbook}. Motivated by this identification, we say that $F$ is a GWHF with twisted kernel $H(z)=H^{GEF}(z)=e^{-|z|^2/2}$.

In the situation of Lemma \ref{le_gef}, the vector \eqref{eq_vector} has a singular covariance, as its third component vanishes. The next proposition shows that up to a change of variables, this is the only case in which the components of \eqref{eq_vector} are deterministically related. 

\begin{prop} \label{prop_singular}
Suppose there exist constants $a, b, c \in \C$,
not all zero, such that, almost surely,
\begin{equation}\label{eq:linrel}
aF^0(0)= b\partial F^0(0) + c \bar \partial F^0(0).
\end{equation}
Then $|c|^2-|b|^2 > 0 $ and 
\begin{align}\label{eq_con}
F^0(z)= e^{[-z(\bar a c+ a \bar b)+\bar z (\bar a b + a \bar c) ]/(|c|^2-|b|^2)} \cdot F^{GEF}(R(z)),
\end{align}
where 
\begin{align}\label{eq_R}
    R(z)= {(|c|^2-|b|^2)^{-1/2}}(cz- b \bar z)
\end{align}
and $F^{GEF}$ is a zero-mean GWHF with twisted kernel $H^{GEF}(z):=e^{-|z|^2/2}$.
\end{prop}
\begin{proof}
The deterministic function $F^1$ plays no role in this proposition, so we will simplify the notation by assuming that $F^1 \equiv 0$ and thus $F = F^0$.

Equation \eqref{eq:linrel} can be rewritten in terms of twisted derivatives as
\[
aF(0) = b\mathcal{D}_1 F(0) + c\mathcal{D}_2 F(0).
\]
Since the stochastics of $F$ are invariant under twisted shifts, the same relation holds when we replace $F$ by $F_z$. This implies that, for each $z \in \C$,
\begin{align}\label{eq_fz}
aF(z) = b\mathcal{D}_1 F(z) + c\mathcal{D}_2 F(z)
\end{align}
holds almost surely (where the exceptional event that may depend on $z$). Multiplying both sides of \eqref{eq_fz} by $\overline{F(0)}$ and taking an expectation yields
\[
aH(z) = b\mathcal{D}_1 H(z) + c\mathcal{D}_2 H(z), \qquad z \in \mathbb{C}.
\]
This is equivalent with 
\begin{equation}\label{eq:H}
(a+b\bar z/2- cz/2)H(z)= b\partial H(z)+ c \bar \partial H(z), \qquad z \in \mathbb{C}.
\end{equation}
Define 
$$
G(z) = H(bz+\bar c \bar z) 
$$
so that by the chain rule
\begin{align}\label{eq_pde}
\partial G(z)= b \partial H(bz + \bar c \bar z) + c \bar \partial H( bz + \bar c \bar z). 
\end{align}
Equation \eqref{eq:H} can be written in terms of $G$ as
\begin{align}\label{eq_pde_G}
\partial G(z) = G(z) \bigg[ a+b \frac{\overline{bz + \bar c \bar z}}{2} - c \frac{bz + \bar c \bar z}{2} \bigg]
=
G(z) \bigg[ a+ \frac{|b|^2-|c|^2}{2} \bar z \bigg]. 
\end{align}
As a consequence of \eqref{eq_Hsym},
\begin{align*}
G(-z)= \overline{G(z)}, \qquad z \in \mathbb{C}.
\end{align*}
Combining this with \eqref{eq_pde_G} we obtain
\begin{align}\label{eq_pde_G2}
-\bar\partial G(-z) = G(-z) \bigg[ \bar a+ \bar b \frac{{bz + \bar c \bar z}}{2} - \bar c \frac{\bar b \bar z + c z}{2} \bigg]. 
\end{align}

We note that the function
\begin{equation*}
\tilde{G}(z)
= e^{2i \mathrm{Im}(az) + (|b|^2-|c|^2)|z|^2/2}, \qquad z \in \mathbb{C},
\end{equation*}
also solves \eqref{eq_pde_G} and \eqref{eq_pde_G2}, while $\tilde{G}(0)= 1 = H(0) = G(0)$. Since $\tilde{G}$ is non-vanishing, we can use the PDEs \eqref{eq_pde_G} and \eqref{eq_pde_G2} to show that $(\partial(G/\tilde{G}), \bar\partial(G/\tilde{G}) )\equiv (0,0)$ and conclude that $G \equiv \tilde{G}$, i.e.,
\begin{equation} \label{eq:Ga}
H(bz+\bar c \bar z)=
e^{2i \mathrm{Im}(az) + (|b|^2-|c|^2)|z|^2/2}, \qquad z \in \mathbb{C}.
\end{equation}
As a consequence,
\begin{equation} \label{eq:H_rel}
|H(bz+ \bar c \bar z)|^2 = 
e^{(|b|^2-|c|^2)|z|^2}, \qquad z \in \mathbb{C}.
\end{equation}
On the other hand, by \eqref{A3} and \eqref{A4}, $|H(z)|^2$ has a strict local maximum at $z=0$, which forces $|c|> |b|$, as claimed (unless $b=c=0$, in which case \eqref{eq:linrel} implies $a=b=c=0$, which we excluded.) 

Let us replace $z$ by $\bar z/(|c|^2-|b|^2)^{1/2}$ in \eqref{eq:Ga} to get
\begin{equation} \label{eq:H_rel2}
H\big(M(z)\big)
= e^{-2i \mathrm{Im}(\overline{a_*} z) -|z|^2/2}, 
\end{equation}
where 
\begin{align*}
M(z)= {(|c|^2-|b|^2)^{-1/2}} (\bar c z+ b \bar z)
\quad\mbox{and}\quad
a_*= {(|c|^2-|b|^2)^{-1/2}}a\,.
\end{align*}
The linear map $M$ is \emph{symplectic}, that is,
\begin{align*}
\Im \big( M(z) \overline{M(w)} \big) = \Im \left(z \bar w\right), \qquad z,w \in \mathbb{C},
\end{align*}
so by \eqref{eq:H_rel2} we obtain
\begin{align*}
\E \Big[ e^{2i \mathrm{Im} (\overline{a_*} z)}F(M(z)) \cdot  \overline{ e^{2i \mathrm{Im}( \overline{ a_*} w)} F(M(w)) } \Big] &= e^{2i \mathrm{Im}(\overline{ a_*}(z-w))} H(M(z-w)) e^{i \Im (z \bar w)}
\\
&= e^{-|z-w|^2/2} e^{i \Im (z \bar w)}.
\end{align*}
Hence,
\begin{align}\label{eq_con2}
F^{GEF}(z):=e^{2i \mathrm{Im}(\overline{a_*} z)} 
F(M(z))
\end{align}
is a GWHF with twisted kernel $H^{GEF}(z):=e^{-|z|^2/2}$. (Clearly, almost all realizations of $F^{GEF}$ are $C^2$ because those of $F$ are.)
We note that the inverse transformation of $M$ is \eqref{eq_R},
and \eqref{eq_con2} then yields \eqref{eq_con}.
\end{proof}

\subsection{Zeros and almost multiplicities}
We first note that the zeros of $F$ are almost surely non-degenerate.
\begin{lemma}\label{lemma_nd}
The following holds almost surely:
\begin{align}\label{eq_nd}
			\det DF(z) \not=0, \qquad z \in \{F=0\}.
		\end{align}
\end{lemma}
\begin{proof}
The claim follows from \cite[Proposition 6.5]{level}, which requires that $F$ be $C^2$ almost surely, as we assume, and that the probability density of $F(z)$ be bounded near the origin uniformly on $z$, which is true because $F(z)$ is a circularly symmetric complex Gaussian vector around its mean, has
	$\var[F(z)] = \sigma^2 H(0)=\sigma^2$, and a bounded mean, by \eqref{A1}. 
\end{proof}
The following proposition gives an analytic expression for the expected charge of a GWHF on a domain, and provides a useful benchmark to numerically test Theorem \ref{th_main}.
\begin{prop}[First intensity of charged zeros]\label{lemma_fi}
Set
\begin{align}\label{eq_fi}
\rho^\pm_1(z) :=  \frac{1}{\pi \sigma^2} e^{-|F^1(z)|^2/\sigma^2 } \Big( |\Dop_1F^1(z)-c_1F^1(z) |^2 - |\Dop_2F^1(z)- c_2F^1(z)|^2 + \sigma^2 \Big)  , \, z \in \C,
\end{align}
where $c_1:=\partial H(0)$ and $c_2:=\bar \partial H(0)$. 
Then, for every Borel set $E \subset \C$,
\begin{align}
\E \Big[\sum_{z\in E, F(z)=0} \charge_z \Big]= \int_E \rho^\pm_1(z)\, dA(z).
\end{align}
Similarly, there exists a measurable function
$\rho_1: \mathbb{C} \to [0,\infty)$ such that
\begin{align}\label{eq_ro1}
\E \big[ \# \{z \in E: F(z)=0\} \big]= \int_E \rho_1(z)\, dA(z).
\end{align}
As a consequence, if $E$ has null Lebesgue measure, then
$\{F=0\} \cap E =  \emptyset$ almost surely.
\end{prop}
\begin{proof}
To obtain the functions $\rho_1$ and $\rho^\pm_1$ we shall invoke Kac-Rice's formulae \cite{level, adler}. Specifically, to obtain a measurable function $\rho_1$
satisfying \eqref{eq_ro1}, we invoke \cite[Theorem 6.2]{level}, which requires: (a) $F$ to be almost surely $C^1$, which we assume; (b) $\var[F(z)]$ to be non-zero for all $z$, which in our case is $\sigma^2$; (c) the non-degeneracy condition \eqref{eq_nd}. (The theorem also gives a concrete expression for $\rho_1$, which we shall not need.) 

For $\rho^\pm_1$ we need a weighted version of Kac-Rice's formula, found in \cite[Theorem 6.4]{level}, and a routine regularization argument, which can be found in the proof of \cite[Lemma 3.3]{hkr22} and we skip. Important for us is that \cite[Theorem 6.4]{level} provides the expression,
\begin{align} \label{eq_rhopm1}
\rho_1^\pm(z)&=\frac{1}{\pi \sigma^2} e^{-|F^1(z)|^2/\sigma^2 } \cdot \mathbb{E} \big[ \det DF(z) \,|\, F(z) =0] \nonumber \\
&= \frac{1}{\pi \sigma^2} e^{-|F^1(z)|^2/\sigma^2 } \cdot \mathbb{E} \big[ \big( |\partial F(z)|^2-|\bar \partial F(z)|^2\big)  \,|\, F(z) =0] \nonumber \\
&=\frac{1}{\pi \sigma^2} e^{-|F^1(z)|^2/\sigma^2 } \cdot \mathbb{E} \big[ \big( |\mathcal D_1 F(z)+ \tfrac{\bar z}{2} F(z)|^2-|\mathcal D_2 F(z)- \tfrac{z}{2} F(z)|^2\big)  \,|\, F(z) =0\big]. 
\end{align} 
Using Lemma \ref{lemma_der} and stochastic invariance under twisted shifts, we write 
$$c_1= \E [\mathcal D_1 F^0(z) \cdot \overline{F^0(z)}]
=\E [\mathcal D_1 F^0(0) \cdot\overline{F^0(0)}]
=\E [\partial F^0(0) \cdot\overline{F^0(0)}]
= \partial H(0),
$$ 
and 
$$c_2= \E [\mathcal D_2 F^0(z) \cdot\overline{F^0(z)}]= 
\E [\mathcal D_2 F^0(0) \cdot\overline{F^0(0)}]
=\E [\bar\partial F^0(0) \cdot\overline{F^0(0)}]
=\bar  \partial H(0).
$$ 
As a consequence of \eqref{eq_Hsym}, we have 
that
\begin{align}\label{eq_variances_a}
|c_1|^2=|c_2|^2.
\end{align}
Similarly, direct computation shows that
\begin{align} \label{eq_variances}
\var \mathcal D_1 F^0(z)=-\Delta H(0)+1/2, \quad 
\var \mathcal D_2 F^0(z)=-\Delta H(0)-1/2;
\end{align}
see \cite[Lemma 2.2]{fhkr} for a detailed calculation.

The constants $c_j$ are defined so that $\mathcal D_jF^0(z) - c_j F^0(z)$ is independent of $F^0(z)$, which facilitates conditioning. (In technical terms, we invoke the so-called  \emph{Gaussian regression formula}, see, e.g., \cite[Proposition 1.2]{level}.) We shall split the conditional expectation in \eqref{eq_rhopm1} into two terms. We first write   
\begin{align*}
&\mathbb{E} \big[ |\mathcal D_1 F(z)+ \tfrac{\bar z}{2} F(z)|^2  \,|\, F(z) =0\big] \nonumber  \\
&= \E \big[\big| \mathcal D_1 F^1(z) + \sigma (\mathcal D_1 F^0(z) -c_1 F^0(z))+c_1 \sigma F^0(z)+ \tfrac{\bar z}{2} F^1(z) + \sigma \tfrac{\bar z}{2} F^0(z)\big|^2 \,|\, F^0(z) =- {F^1(z)}/{\sigma} \big]. 
\end{align*}
Because $\mathcal D_1 F^0(z)- c_1 F^0(z)$ and $F^0(z)$ are independent, the previous conditional expectation equals 
\begin{align*}
&\E \Big[ \big| \mathcal D_1 F^1(z) + \sigma (\mathcal D_1 F^0(z)- c_1 F^0(z))-c_1 F^1(z) \big|^2\Big] \\
&=\big| \mathcal D_1 F^1(z) - c_1 F^1(z)\big|^2 + \sigma^2 \cdot \E \big|\mathcal D_1 F^0(z)- c_1 F^0(z)\big|^2 \\
&= \big| \mathcal D_1 F^1(z) - c_1 F^1(z)\big|^2+ \sigma^2 \cdot \big( \var (\mathcal D_1 F^0(z))- 2 |c_1|^2+ |c_1|^2 \var F^0(z) \big) \\
&= \big| \mathcal D_1 F^1(z) - c_1 F^1(z)\big|^2 + \sigma^2 \cdot \var (\mathcal D_1 F^0(z)) - \sigma^2 |c_1|^2, 
\end{align*}
where for the first equality we are using that $\mathcal D_1 F^0$
and $F^0$ are circularly symmetric, zero mean complex Gaussians. For the third equality, we employed: $\var F^0(z)=H(0)=1$.

We turn to the second term contributing to \eqref{eq_rhopm1}, which is analyzed similarly:
\begin{align*}
&\mathbb{E} \big[ |\mathcal D_2 F(z)- \tfrac{z}{2} F(z)|^2  \,|\, F(z) =0\big] \nonumber  \\
&= \E \Big[\big| \mathcal D_2 F^1(z) + \sigma (\mathcal D_2 F^0(z) -c_2 F^0(z))+c_2 \sigma F^0(z)- \tfrac{z}{2} F^1(z) - \sigma \tfrac{z}{2} F^0(z)\big|^2 \,|\, F^0(z) =- {F^1(z)}/{\sigma} \Big] \\
&= \E \big[ \big| \mathcal D_2 F^1(z) + \sigma (\mathcal D_2 F^0(z)- c_2 F^0(z))-c_2 F^1(z) \big|^2\big] \\
&=\big| \mathcal D_2 F^1(z) - c_2 F^1(z)\big|^2 + \sigma^2 \cdot \E |\mathcal D_2 F^0(z)- c_2 F^0(z)|^2 \\
&= \big| \mathcal  D_2 F^1(z) - c_2 F^1(z)\big|^2 + \sigma^2 \cdot \var (\mathcal D_2 F^0(z)) - \sigma^2 |c_2|^2.
\end{align*}
Combining all the above steps with \eqref{eq_variances_a} and \eqref{eq_variances}, we obtain 
\begin{equation*}
\E \big[ \big| \partial F(z) \big|^2- \big| \bar \partial F(z) \big|^2 \,\big|\, F(z)=0 \big] 
= \big| \mathcal D_1 F^1(z) - c_1 F^1(z)\big|^2 - \big| \mathcal D_2 F^1(z) - c_2 F^1(z)\big|^2 + \sigma^2. \qedhere
\end{equation*}
\end{proof}
\begin{rem}[Universal equilibrium of charge]
When $F^1 \equiv 0$, we have $\rho^\pm_1 \equiv \frac{1}{\pi}$, a fact that was also derived in \cite[Theorem 1.12]{hkr22}, and, interestingly, is \emph{independent of the twisted kernel} $H$.
Remarkably, Proposition \ref{lemma_fi} shows that \emph{the universal equilibrium of charge also holds in the presence of a non-zero mean} $F^1$, as long as the twisted kernel $H$ is real-valued, a situation that is often encountered in applications.
Indeed, when $H$ is real-valued, the identity $\overline{H(z)}= H(-z)$ forces $c_1=c_2=0$, and \eqref{eq_fi} is independent of $H$. 
\end{rem}

We now turn attention to those zeros of a GWHF $F$ that are degenerate in the sense that $DF$ is rank-deficient, or approximately so up to a certain numerical tolerance. We estimate the likeliness of encountering such approximate zeros, as they are a source of challenge for \calgo.

\begin{prop}[Approximately degenerate zeros are unlikely to occur]
\label{prop_gridevent}
Let $\alpha,\beta>0$ and $0<\delta \leq 1$. Then the following hold.
\begin{itemize}
\item[(i)] For each $\lambda \in \mathbb{C}$, the probability that the estimate
\begin{align}\label{eq_lala}
\lvert F_\lambda(0) \rvert \leq \alpha
\quad \mbox{and}\quad
|\det DF_\lambda(0)| \leq \beta
\end{align}
holds is at most $\leq C \alpha^2 \beta \sigma^{-4}$ where the constant $C$ depends only on the twisted kernel $H$ and not on the mean $F^1$.
\item[(ii)] The probability that the estimate \eqref{eq_lala} holds for some $\lambda \in \gridsparse$ is at most $\leq C L^2 \alpha^2 \beta \sigma^{-4}\delta^{-3/2}$,
where $\gridsparse$ is the set \eqref{eq_gs}.
\end{itemize}
(Here, $C$ is a constant that depends only on the twisted kernel $H$.)
\end{prop}
\begin{proof}
We focus on proving part (i). Once this is done, a union bound and \eqref{eq_eg} will yield part (ii). Since the function $F^0$ is invariant under twisted shifts, it suffices to prove (i) only in the case $\lambda=0$. We also assume first that $\sigma=1$. The case of general $\sigma$ is dealt with in the end. We denote by $\Gamma \in \mathbb{C}^{3 \times 3} $ the covariance matrix of the Gaussian random vector  $(F^0(0), \partial F^0(0), \bar \partial F^0(0))$.

\noindent {\bf Step 1}. We start with the observation that the set
\begin{align}
T(u,s):= \{z \in \mathbb{C}:  \big| |z + u|^2 - s \big| \leq \beta \}, \qquad u \in \mathbb{C}, s \in \mathbb{R},
\end{align}
has area
\begin{align}\label{eq_hasa}
\mathrm{Area}(T(u,s)) \leq 2 \pi \beta.   
\end{align}
For negative $s$ the claim is clear. For $s\geq \beta$, we have 
\[\mathrm{Area}\big( T(u,s) \big) = \pi(\beta+s)- \pi (s- \beta)= 2 \pi \beta,
\]
while, for $s < \beta$,
\[
\mathrm{Area}\big( T(u,s) \big) \leq \pi (\beta+s) \leq 2 \pi \beta. 
\]
So, indeed, \eqref{eq_hasa} holds.

\smallskip

\noindent {\bf Step 2}. We start by assuming that $\Gamma$ is non-singular. We denote by  
\begin{equation} \label{eq_means}
(\mu_1, \mu_2, \mu_3):=
(F^1(0), \partial F^1(0), \bar \partial F^1(0))
\end{equation}
the mean of the vector $(F(0), \partial F(0), \bar \partial F(0))$.
Let $a, b>0$ be numbers such that $a\leq \Gamma\leq b$
in the sense of positive definite matrices. In particular, $\det \Gamma \geq a^3$ and
\begin{align*}
\frac{1}{\det \Gamma} e^{-(z_1, z_2, z_3) \Gamma^{-1}(z_1, z_2, z_3)^*} \leq \frac{1}{a^3}
e^{-(|z_1|^2+|z_2|^2+|z_3|^2)/b}
\leq \frac{1}{a^3} e^{-|z_2|^2/b}
, \qquad z_1,z_2,z_3 \in \mathbb{C}.
\end{align*}
Consequently,
\begin{align}
&\mathbb{P}(|F(0)|\leq \alpha, |\det DF(0)| \leq \beta)
\\
&\qquad = \frac{1}{\pi^3 \det \Gamma}\int_{| \mu_1+ z_1| \leq \alpha, ||\mu_2 +z_2|^2-|\mu_3 + z_3|^2| \leq \beta} e^{-(z_1, z_2, z_3) \Gamma^{-1}(z_1, z_2, z_3)^*} dA(z_1) dA(z_2) dA(z_3)\nonumber
\\
&\qquad \leq \frac{1}{\pi^3 a^3}\int_{| \mu_1+ z_1| \leq \alpha, ||\mu_2 +z_2|^2-|\mu_3 + z_3|^2| \leq \beta} e^{-|z_2|^2/b} dA(z_1) dA(z_2) dA(z_3)\nonumber
\\
&\qquad = \frac{1}{\pi^3 a^3} 
\int_{| z_1| \leq \alpha} 
\int_{z_2\in\mathbb{C}} e^{-|z_2-\mu_2|^2/b}
\int_{\big||z_2|^2-|z_3|^2\big| \leq \beta}
\,dA(z_3) dA(z_2) dA(z_1)
\\
&\qquad =
\frac{1}{\pi^3 a^3}
\mathrm{Area}(B_\alpha(0)) \cdot
\int_{z_2\in\mathbb{C}} e^{-|z_2-\mu_2|^2/b} \cdot
\mathrm{Area}\big(T(0,|z_2|^2)\big) \,dA(z_2)
\\
&\qquad\leq C \alpha^2 \beta,
\end{align}
for a constant $C$ independent of $F^1$. As a consequence, the same estimate also holds if we replace $F$ by $F_\lambda$ with $\lambda \in \C.$

\smallskip

\noindent {\bf Step 3}. We now address the case in which $(F^0(0), \partial F^0(0), \bar \partial F^0(0))$ has a singular covariance. This means that there exist constants $a,b,c$ such that 
\[
aF^0(0)=b\partial F^0(0) + c \bar \partial F^0(0). 
\]
By Proposition \ref{prop_singular}, we have $|c|^2-|b|^2>0$, so we can normalize the constants so that $|c|^2-|b|^2=1.$
By the same Proposition, we have 
\[
F^0(z)= 
e^{\overline{ a_*}  z- a_* \bar z}F^{GEF}\big(cz-b \bar z\big), 
\]
where $F^{GEF}$ is a GWHF with twisted kernel $H^{GEF}(z)=e^{-|z|^2/2}$ and $a_*=-a \bar c-\bar a b.$
By Lemma \ref{le_gef}, $\bar \partial F^{GEF}(0)=0$, and therefore the Jacobian determinant of $F$ at the origin is 
\begin{align*}
\det DF(0)&=|\partial F(0)|^2- |\bar \partial F(0)|^2 \\
& = \big| \partial F^1(0) + \overline{a_*}  F^{GEF}(0) + c \partial F^{GEF}(0)\big|^2- \big| \bar \partial F^1(0) - a_*  F^{GEF}(0) - b \partial F^{GEF}(0)\big|^2,
\end{align*}
where we used the chain rule for Wirtinger derivatives. The vector $(F^{GEF}(0), \partial F^{GEF}(0))$ is, by Lemma \ref{le_gef}, standard complex Gaussian, i.e., the entries are independent and $\mathcal{N}_{\mathbb{C}}(0,1)$ distributed. The density of this Gaussian vector is bounded above by $\frac{1}{\pi^2}$, so we may estimate  
\begin{equation} \label{eq_almost_deg_est}
\mathbb{P}(|F(0)|\leq \alpha, |\det DF(0)| \leq \beta)
\leq \frac{1}{\pi^2}\int_S dA(z_1) dA(z_2),
\end{equation}
where 
\[
S:= \left\{(z_1, z_2) \in \mathbb{C}^2\,:\,
|z_1+\mu_1| \leq \alpha, ~~\big| | \mu_2+ \overline{a_*} z_1+ c z_2|^2  - |\mu_3 - a_* z_1- b z_2|^2 \big| \leq \beta \right\}
\]
and $\mu_1, \mu_2, \mu_3$ are defined as in \eqref{eq_means}. Using that $|c|^2-|b|^2=1$, the second inequality in the definition of $S$ can be rewritten as  
\[
\big||z_2 + u|^2 - s \big|\leq \beta,
\]
where 
\begin{align*}
u= u(z_1) &=  \bar c(\mu_2+ \overline{a_*} z_1)+ \bar b(\mu_3-a_* z_1),\\ 
s = s(z_1) &= |u|^2-  |\mu_2+ \overline{a_*} z_1|^2+ |\mu_3- a_* z_1|^2.
\end{align*}
Recalling \eqref{eq_almost_deg_est} and \eqref{eq_hasa} we conclude:
\[
\mathbb{P}(|F(0)|\leq \alpha, |\det DF(0)| \leq \beta) \leq \frac{1}{\pi^2} \int_{|z_1+ \mu_1|\leq \alpha} \int_{ z_2 \in T(u(z_1),s(z_1))} dA(z_2) dA(z_1) \leq 2 \alpha^2 \beta.
\]
Because the previous estimate is independent of $F^1$, it also holds if we replace $F$ by $F_\lambda$ for arbitrary $\lambda \in \C$. We have thus proved the theorem in the case $\sigma=1$. 

\medskip

\noindent {\bf Step 4}. 
 Finally, we will eliminate the assumption $\sigma=1$ and deal with the general case $\sigma>0$. Applying result for $\sigma=1$ to the function $z \mapsto \frac{1}{\sigma} F^1(z) +F^0(z)$, we obtain
 \begin{align*}
&\mathbb{P}(|F(0)|\leq \alpha, |\det DF(0)| \leq \beta) \\
&= \mathbb{P}\bigg(|F^1(0)/\sigma + F^0(0)|\leq \alpha/\sigma, |\det D\big( \frac{1}{\sigma}F^1+ F^0\big)(0)| \leq \beta/\sigma^2\bigg)
\leq C \alpha^2 \beta/ \sigma^4,
 \end{align*}
 where $C$ depends only on the twisted covariance kernel $H$. 
\end{proof}

\subsection{Poincar\'{e} index}
We now note that Poincar\'{e}'s index formula is applicable to GWHFs; see \cite{fhkr} for a comprehensive study of Poincar\'{e} indices of GWHFs and applications to the classification of zeros.
\begin{lemma}\label{le_p}
Let $\Omega \subset \mathbb{C}$ be a square or a disk. Then, almost surely,
\begin{align}\label{eq_Poincare}
\int_{\partial \Omega} \frac{dF}{F}
		= \int_{\partial \Omega} \left(\frac{\partial F }{F} \,dz + \frac{\bar{\partial} F }{F} \,d\bar{z}\right) = 2\pi i \sum_{z \in \Omega, F(z)=0} \charge_z.
\end{align}
\end{lemma}
\begin{proof}
Poincar\'{e}'s index formula \eqref{eq_Poincare} 
is valid if (a) $F$ is twice continuously differentiable, as we assume; (b) $F$ does not vanish on $\partial \Omega$; and (c) $0$ is a regular value of $F$; see, for example, \cite{MR1487640, MR0209411}. Condition (c) holds by Lemma \ref{lemma_nd}, while condition (b) follows from Proposition \ref{lemma_fi}, since $\partial\Omega$ has null Lebesgue measure, and therefore \eqref{eq_ro1} shows that $\E \big[\#\{F=0\} \cap \partial \Omega \big]=0$. 
\end{proof}

\section{Proof of Theorem \ref{th_main}}\label{sec_proof_main}

\subsection{Excluding bad events}
Since $L \geq 1$, we can choose the constant $C>0$ in \eqref{eq_fp} so large that the bound on the failure probability becomes trivial for $\delta \geq 1/64$. Hence, we assume without loss of generality that $0<\delta<1/64$. As a consequence,
\[\sqrt{\delta} \leq \big\lceil{{\delta}^{-1/2}}\big\rceil \delta = \ccdelta \leq 2 \sqrt{\delta} \leq 1/4,\] an estimate that we will use repeatedly without further mention. We will exclude a number of events that have a small probability for small values of $\delta$. It will be convenient to use the notation
$\gamma = \cons \sigma$ where $\cons>0$ is a constant to be fixed shortly, that depends only on the twisted kernel $H$, cf. \eqref{eq_H}.

We start with the excursion bounds in Lemma \ref{lem:excu} and consider the constants $c,c',C>0$ of its statement. Firstly, we invoke Lemma \ref{lem:excu} with $s=\gamma \sqrt{\log \tfrac{1}{\delta}}$ 
and take $\cons \geq c^{-1/2}$
to conclude that
\begin{align}\label{eq_e1}
\sup_{|z| \leq 10, |w| \leq 9L}
|\partial^\alpha F_w(z)| \leq \gamma \sqrt{\log \tfrac{1}{\delta}}, \quad |\alpha|=1,2,
\end{align}
except on an event of probability at most 
\begin{align*}
9^2 C L^2 \exp\Big(c'\frac{A^2}{\sigma^2}\Big)
\exp\Big({-}c\frac{s^2}{\sigma^2}\Big)
&=9^2 C L^2 \exp\Big(c'\frac{A^2}{\sigma^2}\Big)
\exp\Big({-}c\cons^2\log\tfrac{1}{\delta}\Big)
\\
&\leq 9^2 C L^2 \exp\Big(c'\frac{A^2}{\sigma^2}\Big)
\delta.
\end{align*}
(We choose $9L$ in \eqref{eq_e1} because $Q_{L+5}(0) \subset D_{9L}(0)$.)
We use \eqref{eq_e1} to bound the error in first order Taylor expansions. By taking $\cons$ suitably large we further ensure that
\begin{align}\label{eq:lineariz}
	&\big| F_w(z+u) - F_w(u)- \partial F_w(u) z- \bar \partial F_w(u) \bar z \big| \leq \gamma \sqrt{\log \tfrac{1}{\delta}} \cdot |z|^2, \quad |z|,|u| \leq 5, |w|_\infty \leq L+5,
	\\\label{eq:lineariz2}
	&\big| T F_w(z+u) - T F_w(u)\big| \leq \gamma \sqrt{\log \tfrac{1}{\delta}} \cdot |z|, \quad T=\partial_x,\partial_y,\partial,\bar\partial,\, |z|,|u| \leq 5, |w|_\infty \leq L+5.
\end{align}
Secondly, we use Proposition \ref{prop_gridevent} with
\begin{align*}
	\alpha = \cons_0 \gamma \sqrt{\log \tfrac{1}{\delta}} \delta, \qquad
	\beta = \cons_1 \gamma^2 \log \tfrac{1}{\delta} \sqrt{\delta}
\end{align*}
and constants $\cons_0$ and $\cons_1$ to be fixed, to ensure that
\begin{equation}
	\mbox{for all }\lambda \in \gridsparseno_{L+5}:\quad
	|F_\lambda(0)| \geq \cons_0 \gamma \sqrt{\log \tfrac{1}{\delta}} \delta \quad \text{or} \quad |\det DF_\lambda(0)| \geq \cons_1 \gamma^2 \log \tfrac{1}{\delta} \sqrt{\delta}
	\label{eq:fdfjointsmall}
\end{equation}
holds, after excluding an event with probability
\[\lesssim \gamma^4 \sigma^{-4} \cons_0^2 \cons_1 L^2 (\log \tfrac{1}{\delta})^2\cdot \delta\lesssim \cons^4 \cons_0^2 \cons_1 L^2 (\log \tfrac{1}{\delta})^2\cdot \delta.\]

Finally, we exclude certain null-probability events. First, by Proposition \ref{lemma_fi}, we can exclude the null-probability event that $F$ vanishes on any grid point $\lambda \in \Lambda$ or on any point of the boundary of the test boxes, i.e., $\partial Q_{\ccdelta}(\lambda)$, $\lambda \in \Lambda_L$. This implies, in particular, that in Step 2 of the algorithm, $F_\lambda(\mu_j) \not=0$ and $\theta_\lambda$ is always defined by \eqref{eq_tl}. Second, by Lemma \ref{lemma_nd}, after excluding a null-probability event, we have
\begin{align}\label{eq_nmz}
\det DF_\lambda(z) \not=0,\mbox{ whenever }F_\lambda(z)=0\mbox{ and }\lambda \in \Lambda.
\end{align}
Third, applying Lemma \ref{le_p} to the domain $\Omega=Q_{\ccdelta}(0)$, we can exclude a null-probability event and ensure that
\eqref{eq_Poincare} holds simultaneously for all functions $F_\lambda$, $\lambda \in \grid$ in lieu of $F$.

Altogether, we exclude events with total probability $\lesssim L^2 \cdot (\log \tfrac{1}{\delta})^2 \cdot \exp\big(c' {A^2}/{\sigma^2}\big) \cdot \delta$. We now show that under the complementary event, the conclusions of the theorem hold. (To fully specify the failure probability, we need to fix the constants $\cons_0$ and $\cons_1$; for clarity, we do this throughout the proof, by imposing various absolute lower bounds on $\cons_0$ and $\cons_1$.)

\subsection{The true zeros are adequately separated}

We claim that the zero set of $F$ is $7\ccdelta=7 \sdelta \delta$ separated (wrt.\ the infinity norm):
\begin{align}\label{eq_7sep}
	\inf \big\{ |\xi-\xi'|_\infty: \xi, \xi' \in \{F=0\} \cap \domainplus,\, \xi\not=\xi'\big\} \geq 7 \sdelta \delta.
\end{align}
Suppose on the contrary that $F(\xi_1)=F(\xi_2)=0$, with $\xi_1, \xi_2 \in \domainplus$ and \[0<|\xi_1- \xi_2|_\infty < 7 \sdelta \delta.\]
We linearize $F_{\xi_1}$ around $0$ and evaluate at $\xi_2-\xi_1$:
by \eqref{eq:lineariz} with $w=\xi_1$, $u=0$ and $z=\xi_2-\xi_1$, we obtain
\begin{equation} \label{eq:1}
	\big|\partial F_{\xi_1}(0)(\xi_2-\xi_1) + \bar \partial F_{\xi_1}(0)  \overline{(\xi_2-\xi_1)}
	\big| \leq \gamma \sqrt{\log \tfrac{1}{\delta}} |\xi_2-\xi_1|^2.
\end{equation}
The use of \eqref{eq:lineariz} is valid because
$|\xi_1|_\infty \leq L+\ccdelta\leq L+1$ and $|\xi_2-\xi_1| \leq 7\sqrt{2} \ccdelta \leq 7\sqrt{2}/4 < 5$, and we also have $|F_{\xi_1}(0)|=|F(\xi_1)|=0$ and $|F_{\xi_1}(\xi_2-\xi_1)|=|F(\xi_2)|=0$.

Let $v := (\xi_2-\xi_1)/|\xi_2-\xi_1|$. Regarding $v$ as an element of $\mathbb{R}^2$, \eqref{eq:1} implies
\begin{equation}
\label{eq:absdfv}
\big\lvert DF_{\xi_1}(0) \cdot v \big\rvert
\leq \gamma \sqrt{\log \tfrac{1}{\delta}} |\xi_2-\xi_1|
\leq
7 \sqrt{2} \gamma \sqrt{\log \tfrac{1}{\delta}} \sdelta \delta.
\end{equation}
We now argue that one can move away from $\xi_1$ along the line with direction $v$, traveling a distance $\gtrsim \sqrt{\delta}$, and finishing close to the sparse grid $\gridsparse$. More precisely, we use Lemma \ref{lemma_sparse} (with $L+1$ in lieu of $L$) to
select the point  
$\xi_3 = \xi_1 +b \sqrt{\delta} v$ 
with
\begin{equation}
2 \leq b \leq 4,
\end{equation}
and $\tau \in \gridsparseno_{L+2}$ with
\begin{equation}\label{eq_newq1}
|\xi_3-\tau| \leq \delta.
\end{equation}
By \eqref{eq:lineariz}, we have
$$
|F_{\xi_1}(\xi_3-\xi_1)-\underbrace{F_{\xi_1}(0)}_{=0}- \partial F_{\xi_1}(0)(\xi_3-\xi_1)-\bar \partial F_{\xi_1}(0)\overline{(\xi_3-\xi_1)} | \leq \gamma \sqrt{ \log\tfrac{1}{\delta}} |\xi_1-\xi_3|^2
\leq  b^2 \gamma \sqrt{ \log\tfrac{1}{\delta}} \delta,
$$
which, together with \eqref{eq:absdfv}, implies
\begin{align*}
	|F_{\xi_3}(0)| 
	& =  |F_{\xi_1}(\xi_3 - \xi_1)|  
	\\
	& \leq b^2 \gamma \sqrt{ \log\tfrac{1}{\delta}} \delta 
	+ |\partial F_{\xi_1}(0)(\xi_3-\xi_1)+\bar \partial F_{\xi_1}(0)\overline{(\xi_3-\xi_1)} |
	\\
	& = b^2 \gamma \sqrt{ \log\tfrac{1}{\delta}} \delta
	+ \big\lvert b \sqrt{\delta} DF_{\xi_1}(0) \cdot v \big\rvert
	\\
	& \leq b^2 \gamma \sqrt{ \log\tfrac{1}{\delta}} \delta +
	7 \sqrt{2} b \gamma \sqrt{\log \tfrac{1}{\delta}} \sdelta \delta^{3/2}
	\\
	& \leq \big(b^2+28b \big) \gamma \sqrt{\log\tfrac{1}{\delta}} \delta
    \\
    & \leq 128 \gamma \sqrt{\log\tfrac{1}{\delta}} \delta.
\end{align*}
Recall that, by construction, $
\tau \in \gridsparseno_{L+2} \subset \Omega_{L+2}$, so we can invoke \eqref{eq:lineariz} and \eqref{eq_newq1} to obtain, 
\begin{equation} \label{eq:tau_star}
	\big| F_{\tau}(\xi_3-\tau)- F_{\tau}(0) - \partial F_\tau(0)(\xi_3-\tau)- \bar \partial F_{\tau}(0)\overline{(\xi_3-\tau)} \big| 
	\leq \gamma \sqrt{\log \tfrac{1}{\delta}} \delta^2. 
\end{equation}
By \eqref{eq_e1}, we have
\begin{align}\label{eq_a1}
|\bar\partial F_\tau(0)|,
|\partial F_\tau(0)| \leq \gamma \sqrt{\log \tfrac{1}{\delta} }.
\end{align}
On the other hand, we have just derived the estimate
\begin{align*}
	|F_{\tau}(\xi_3-\tau)|=
|F_{\xi_3}(0)|\leq 128 \gamma \sqrt{\log\tfrac{1}{\delta}} \delta.
\end{align*}
Combining this with \eqref{eq_newq1}, \eqref{eq:tau_star} and \eqref{eq_a1}, we have
\begin{align}\label{eq_m0}
\begin{aligned}
|F_{\tau}(0)| 
&\leq 128 \gamma \sqrt{\log\tfrac{1}{\delta}} \delta + 2  \gamma \sqrt{\log \tfrac{1}{\delta} } \delta + \gamma \sqrt{\log \tfrac{1}{\delta}} \delta^2
\leq 131 \gamma \sqrt{\log\tfrac{1}{\delta}} \delta,
\end{aligned}
\end{align}
where we used $\delta \leq 1$.
Next, we evaluate the directional derivative from $\tau$ towards $\xi_1$.
Note first that 
\begin{align}\label{eq_veryfa}
|\tau-\xi_1| \leq |\tau-\xi_3|+|\xi_3-\xi_1|
\leq \delta + b \sqrt{\delta} \leq 
(b+1) \sqrt{\delta} \leq 5\sqrt{\delta}.
\end{align}
By \eqref{eq:lineariz},
$$
|\underbrace{F_{\tau}(\xi_1-\tau)}_{=0} - F_{\tau}(0) - \partial F_{\tau}(0)(\xi_1-\tau)- \bar \partial F_{\tau}(0)\overline{(\xi_1-\tau)} |
\leq 25\gamma  \sqrt{\log \tfrac{1}{\delta}} \delta,
$$
which, together with \eqref{eq_m0} implies
\begin{align}\label{eq_m1}
|\partial F_{\tau}(0)(\xi_1-\tau)+ \bar \partial F_{\tau}(0)\overline{(\xi_1-\tau)}| \leq 156 \gamma  \sqrt{\log \tfrac{1}{\delta}} \delta.
\end{align}
Second, by \eqref{eq_newq1},
\[|\tau-\xi_1| \geq |\xi_3-\xi_1| - |\tau-\xi_3| \geq b \sqrt{\delta}- \delta \geq \sqrt{\delta},\]
where we used that $b \geq 2$. Set \[w:= (\tau-\xi_1)/{|\tau-\xi_1|}.\]
Then \eqref{eq_m1} implies that
\begin{align}\label{eq_m2}
	|\partial_{w} F_\tau(0)|=
|DF_{\tau}(0) \cdot w | \leq 156 \gamma  \sqrt{\log \tfrac{1}{\delta}} \sqrt{\delta}.
\end{align}
On the other hand, \eqref{eq_e1} implies that
\begin{equation} \label{eq:derbound}
	|\partial_{-iw} F_\tau(0)| \leq 2 \gamma \sqrt{ \log \tfrac{1}{\delta}}.
\end{equation}
Combining \eqref{eq_m2} with \eqref{eq:derbound} and
 the expression for $\det DF_{\tau}(0)$ in \eqref{eq_jacv}, we obtain
$$
|\det DF_{\tau}(0)| \leq 312 \gamma^2 \log \tfrac{1}{\delta} \sqrt{\delta}. 
$$

Taking into account \eqref{eq_m0} we conclude the following: under the assumption that \eqref{eq_7sep} fails, there exists a sparse grid point
$\tau \in \gridsparseno_{L+2}$ such that 
\begin{align*}
|F_{\tau}(0)| \leq C \gamma \sqrt{\log\tfrac{1}{\delta}} \delta
\quad \mbox{and}\quad
|\det DF_{\tau}(0)| \leq C \gamma^2 \log \tfrac{1}{\delta} \sqrt{\delta},
\end{align*}
where $C$ is an absolute constant. This contradicts \eqref{eq:fdfjointsmall} provided that we let the constants $\cons_0$ and $\cons_1$ be large enough.

\subsection{For the points selected in Step 1, the discretely computed change of the argument equals the actual change}

We show that for $\lambda \in \outseta$, the number $\theta_\lambda$,
defined in \eqref{eq_tl}, is
\begin{align}\label{eq_arg}
2\pi i \cdot \theta_\lambda = \int_{\partial Q_{\ccdelta}(0)}  \frac{dF_\lambda(z)}{F_\lambda(z)}.
\end{align}
Recall that we have excluded the event where $F$ vanishes on any point of $\partial Q_{\ccdelta}(\lambda)$. We use the notation of Step 2.
Let $\alpha_j:[0,1] \to \C$, $j=1,\ldots,N$ be the curve
\begin{align*}
\alpha_j(t) := F_\lambda\big(t\mu_{j} + (1-t) \mu_{j-1}\big).
\end{align*}
The concatenation $\alpha:=\alpha_1 * \ldots * \alpha_N$ is a (possibly non-injective) parametrization of $F_\lambda(\partial Q_{\ccdelta}(0))$, where $\partial Q_{\ccdelta}(0)$ is parametrized anticlockwise. Thus, $\alpha$ is a closed curve that avoids $0$ and
\begin{align*}
i\sum_{j=1}^N \Im \int_{\alpha_j} \frac{dz}{z} = i\Im \int_{\alpha} \frac{dz}{z} = 
\int_{\alpha} \frac{dz}{z}= \int_{\partial Q_{\ccdelta}(0)}  \frac{dF_\lambda(z)}{F_\lambda(z)}.
\end{align*}
Inspecting the definition \eqref{eq_tl}, it suffices to show that
\begin{align*}
\Im \int_{\alpha_j} \frac{dz}{z} = \arg \big[F_\lambda(\mu_{j}) \overline{F_\lambda(\mu_{j-1})}\big], \qquad j=1,\ldots,N.
\end{align*}
To this end, we fix $j \in \{1,\ldots,N\}$ and invoke Lemma \ref{lemma_arg} with $\tau = F_\lambda(\mu_{j-1})$. This is possible because we have excluded the event where $F_\lambda$ vanishes on $\partial Q_{\ccdelta}(0)$. It suffices to check that
\begin{align}\label{eq_its}
\Re \big[ F_\lambda(t\mu_{j} + (1-t) \mu_{j-1}) \overline{F_\lambda(\mu_{j-1})} \big] >0, \qquad t \in [0,1].
\end{align}
The segment determined by $\mu_{j-1}$ and $\mu_{j}$ is either vertical or horizontal. Without loss of generality we assume that
it is horizontal: $\Im \mu_{j-1} = \Im \mu_{j}$. Let $t \in [0,1]$. We invoke \eqref{eq:lineariz} with $w=\lambda$ and
\begin{align*}
z_t=t(\mu_{j} - \mu_{j-1}),
\end{align*}
which is allowed because $|z_t| \leq \delta \leq 1$ and $\lambda \in \Lambda_{L}$. Since $z_t \in \mathbb{R}$, the linearization estimate reads
\begin{align}\label{eq_m4}
|F_{\lambda}(\mu_{j-1}+z_t) - F_\lambda(\mu_{j-1}) - \partial_x F_\lambda(\mu_{j-1}) z_t| \leq \gamma \sqrt{\log\tfrac{1}{\delta}} |z_t|^2, \qquad t \in [0,1],
\end{align}
and, by \eqref{eq:lineariz2},
\begin{align}\label{eq_m4p}
	|\partial_x F_{\lambda}(\mu_{j-1}+u) - \partial_x F_\lambda(\mu_{j-1})| \leq \gamma \sqrt{\log\tfrac{1}{\delta}} |u|, \qquad |u|<1.
\end{align}
In the current notation, the fact that $\lambda$ passed the algorithm's selection step (Step 1) gives
\begin{align*}
2|F_{\lambda}(\mu_{j-1})-F_{\lambda}(\mu_j)| \leq |F_\lambda(\mu_{j-1})|.
\end{align*}
(Note that the points $\mu_{j-1},\mu_j$ are indeed as required in \eqref{st12}, because they are  neighboring boundary grid points, successive according to their principal arguments, and thus $|\mu_j-\mu_{j-1}|=|\mu_j-\mu_{j-1}|_\infty=\delta$; this also holds
for $j=1$ because $\mu_0=\mu_N$.)

Since $|z_1|=\delta$, the previous estimate, together with \eqref{eq_m4} gives
\begin{align}\label{eq_m5}
\begin{aligned}
\delta |\partial_x F_\lambda(\mu_{j-1})| &\leq
|F_{\lambda}(\mu_{j-1})-F_{\lambda}(\mu_j)|
+ |F_{\lambda}(\mu_j) - F_\lambda(\mu_{j-1}) - \partial_x F_\lambda(\mu_{j-1}) z_1|
\\
&\leq \frac12 |F_\lambda(\mu_{j-1})| + \gamma \sqrt{\log\tfrac{1}{\delta}} \delta^2.
\end{aligned}
\end{align}
We claim that
\begin{align}\label{eq_claim}
\frac{1}{2} \lvert F_\lambda(\mu_{j-1}) \rvert > 2 \gamma \sqrt{ \log \tfrac{1}{\delta}} \delta^2.
\end{align} 
Suppose not, that is,
\begin{align}\label{eq_supnot}
\lvert F_{\lambda}(\mu_{j-1}) \rvert \leq 4 \gamma \sqrt{ \log \tfrac{1}{\delta}} \delta^2.
\end{align}
Then \eqref{eq_m5} yields
$|\partial_x F_\lambda(\mu_{j-1})| \leq 3 \gamma \sqrt{\log\tfrac{1}{\delta}} \delta$. By Lemma \ref{lem_dervery} and \eqref{eq_supnot},
\begin{align}
|\partial_x F_{\lambda+\mu_{j-1}}(0)| &\leq | \mu_{j-1}|_\infty \cdot |F_{\lambda}(\mu_{j-1})| + 
|\partial_x F_{\lambda}(\mu_{j-1})|
\\
&\leq \ccdelta 4 \gamma \sqrt{ \log \tfrac{1}{\delta}} \delta^2 + 3 \gamma \sqrt{\log\tfrac{1}{\delta}} \delta \leq 7 \gamma \sqrt{\log\tfrac{1}{\delta}} \delta,
\end{align}
while, by \eqref{eq_e1},
$|\partial_y F_{\lambda+\mu_{j-1}}(0)| \leq \gamma \sqrt{\log\tfrac{1}{\delta}}$.
Consequently
\begin{align}\label{eq_contra1}
|\det DF_{\lambda+\mu_{j-1}}(0)| \leq |\partial_x F_{\lambda+\mu_{j-1}}(0)| \cdot |\partial_y F_{\lambda+\mu_{j-1}}(0)|
\leq 7 \gamma^2 {\log\tfrac{1}{\delta}} {\delta}
\leq 7 \gamma^2 {\log\tfrac{1}{\delta}} {\sqrt{\delta}},
\end{align}
cf. \eqref{eq_jacv}. In addition, using again \eqref{eq_supnot},
\begin{align}\label{eq_contra2}
\lvert F_{\lambda+\mu_{j-1}}(0) \rvert 
=\lvert F_{\lambda}(\mu_{j-1}) \rvert \leq 4 \gamma \sqrt{ \log \tfrac{1}{\delta}} \delta^2 
 \leq 4 \gamma \sqrt{ \log \tfrac{1}{\delta}} \delta. 
\end{align}
Since $\lambda \in \outseta \subset \cgrid_L$, $\mu_{j-1} \in  \grid$,
and $|\mu_{j-1}|_\infty=\ccdelta$, we have that
$\lambda+\mu_{j-1} \in \gridsparseno_{L+5}$, cf. \eqref{eq_gs}.
Thus, the combination of \eqref{eq_contra1} and \eqref{eq_contra2} contradicts \eqref{eq:fdfjointsmall} provided that the constants $\cons_0$ and $\cons_1$ are chosen large enough. 
(Note that in the last bounds we gave up factors of order $\sqrt{\delta}$ and $\delta$; however other parts of our argument do not allow us to make use of this.)

Hence we conclude that \eqref{eq_claim} holds. We now combine
\eqref{eq_m4}, \eqref{eq_m5} and \eqref{eq_claim} to estimate 
\begin{align*}
|F_{\lambda}(\mu_{j-1}+z_t) - F_\lambda(\mu_{j-1})| &\leq |F_{\lambda}(\mu_{j-1}+z_t) - F_\lambda(\mu_{j-1}) - \partial_x F_\lambda(\mu_{j-1}) z_t| + |z_t| |\partial_x F_\lambda(\mu_{j-1})| 
\\
&\leq \gamma \sqrt{\log\tfrac{1}{\delta}} |z_t|^2 + \delta |\partial_x F_\lambda(\mu_{j-1})| 
\\
&\leq \gamma \sqrt{\log\tfrac{1}{\delta}} \delta^2 + \frac12 |F_\lambda(\mu_{j-1})| + \gamma \sqrt{\log\tfrac{1}{\delta}} \delta^2
\\
&<|F_\lambda(\mu_{j-1})|.
\end{align*}
 Taking squares yields
\begin{align*}
|F_\lambda(\mu_{j-1})|^2 &> |F_{\lambda}(\mu_{j-1}+z_t) - F_\lambda(\mu_{j-1})| ^2 
\\
&= |F_{\lambda}(\mu_{j-1}+z_t)|^2 + |F_\lambda(\mu_{j-1})|^2 - 2 \Re \big[ F_{\lambda}(\mu_{j-1}+z_t) \overline{F_{\lambda}(\mu_{j-1})}\big],
\end{align*}
and so $\Re \big[ F_{\lambda}(\mu_{j-1}+z_t) \overline{F_{\lambda}(\mu_{j-1})}\big] > \frac12 |F_\lambda(\mu_{j-1}+z_t)|^2>0$, because we $F_\lambda$ does not vanish on $\partial Q_{\ccdelta}(\lambda)$. This gives \eqref{eq_its} and concludes the proof of \eqref{eq_arg}.

\subsection{Each true zero is close to a grid point selected in Step 1}

We show that 
\begin{align}\label{claim1}
	\{F=0\} \cap \domain \subset \outseta + Q_{\sdeltaa \delta}(0)\,.
\end{align}
Let $\zeta \in \{F=0\} \cap \domain$ and $\lambda \in \cgrid_L$ be such that $|\lambda-\zeta|_\infty \leq \sdeltaa \delta= \frac12 \delta^*$ 
(The assumption that $L/\cdelta \in \mathbb{N}$ is important here, so as to be able to choose an adequate $\lambda$ with $|\lambda|_\infty\leq L$, even if $\zeta$ is close to the boundary of $\Omega_L$.)
Let us show that $\lambda$ passes the test in Step 1. Specifically, let $\mu,\mu' \in \Lambda$ with \eqref{st12} and let us show that \eqref{st11} holds.
Note that $\lambda+\mu, \lambda+\mu' \in \gridsparseno_{L+1}$.

As for the right-hand side of \eqref{st11}, by \eqref{eq_e1}, we have
\begin{align}\label{eq_b1}
\begin{aligned}
|F_\lambda(\mu)-F_\lambda(\mu')| &\leq 2\max\Big\{\sup_{|w|\leq 1} |\partial F_\lambda(w)|, \sup_{|w|\leq 1} |\overline{\partial} F_\lambda(w)|\Big\} |\mu-\mu'|
\\
&\leq 2 \gamma \delta \sqrt{\log\tfrac1\delta}.
\end{aligned}
\end{align}
To prove \eqref{st11}, we distinguish two cases.

Suppose first that $|\det DF_{\lambda+\mu}(0)| < \cons_1 \sqrt{\delta} \gamma^2 \log \tfrac{1}{\delta}$. Since $\lambda+\mu \in \gridsparseno_{L+1}$, by \eqref{eq:fdfjointsmall},
\begin{align*}
	|F_\lambda(\mu)|=|F_{\lambda+\mu}(0)| \geq \cons_0 \gamma \delta \sqrt{\log\tfrac1\delta}.
\end{align*}
Comparing this to \eqref{eq_b1} we obtain
\begin{align*}
|F_\lambda(\mu)| \geq \tfrac{1}{2} \cons_0|F_\lambda(\mu)-F_\lambda(\mu')|.
\end{align*}
Thus \eqref{st11} holds, provided that $\cons_0 \geq 4$.

Suppose now that 
\begin{align}\label{supnow}
|\det DF_{\lambda+\mu}(0)| \geq \cons_1 \sqrt{\delta} \gamma^2 \log \tfrac{1}{\delta},
\end{align}
and consider the following estimates on the displacement $\zeta-(\lambda+\mu)$:
\begin{align}\label{r1}
|\zeta-(\lambda+\mu)| \leq |\mu| + |\zeta-\lambda|
\leq \sqrt{2} \sdelta \delta + \sqrt{2} 
\lceil \delta^{-1/2}/4 \rceil \delta
\leq \Big(2 \sqrt{2}+\tfrac{\sqrt{2}}{2}\Big) \sqrt{\delta},
\\\label{r2}
|\zeta-(\lambda+\mu)| \geq |\mu| - |\zeta-\lambda|
\geq \sdelta \delta - \tfrac{\sqrt{2}}{2} \delta^{1/2} 
\geq \Big(1-\tfrac{\sqrt{2}}{2}\Big) \sqrt{\delta}.
\end{align}
(Here, we used that
$\lceil \delta^{-1/2}/4 \rceil \leq \delta^{-1/2}/4 +1 \leq \delta^{-1/2}/2$, because $\delta < 1/16$.)

Let $v := [\zeta-(\lambda+\mu)] / |\zeta-(\lambda+\mu)|$.
The expression for $\det DF_{\lambda+\mu}(0)$ in \eqref{eq_jacv}, together with \eqref{supnow} 
and \eqref{eq_e1} give
\begin{align*}
	\cons_1 \sqrt{\delta} \gamma^2 \log \tfrac{1}{\delta} &\leq
|\det DF_{\lambda+\mu}(0)| \leq
|\partial_v F_{\lambda+\mu}(0)| \cdot |{\partial_{-iv}F_{\lambda+\mu}(0)}|
\\
&\leq C^{-1} \gamma \sqrt{\log\tfrac1\delta} \cdot |\partial_v F_{\lambda+\mu}(0)|,
\end{align*}
for an absolute constant $C>0$. Consequently
\begin{align}\label{pl}
|\partial_v F_{\lambda+\mu}(0)| \geq C \gamma \cons_1 \sqrt{\delta} \sqrt{\log\tfrac1\delta}.
\end{align}
On the other hand, since $|F_{\lambda+\mu}(\zeta-(\lambda+\mu))|=|F(\zeta)|=0$,
\eqref{eq:lineariz} gives
\begin{align*}
	|F_\lambda(\mu)|&=|F_{\lambda+\mu}(0)| =
	|F_{\lambda+\mu}(0) - F_{\lambda+\mu}(\zeta-(\lambda+\mu))|
	\\
	&\geq 
    \big|\partial F_{\lambda+\mu}(0) (\zeta-(\lambda+\mu))+\bar\partial F_{\lambda+\mu}(0) \overline{(\zeta-(\lambda+\mu))}\big| 
    - \gamma  \sqrt{\log\tfrac1\delta}\cdot \big|\zeta-(\lambda+\mu)\big|^2
	\\
	&=\big|\partial_v F_{\lambda+\mu}(0) \big|\cdot \big|\zeta-(\lambda+\mu)\big| - \gamma  \sqrt{\log\tfrac1\delta} \cdot \big|\zeta-(\lambda+\mu)\big|^2.
\end{align*}
Further applying \eqref{pl}, \eqref{r1} and \eqref{r2} gives
\begin{align*}
|F_\lambda(\mu)| &\geq C \gamma \cons_1 \sqrt{\delta} \sqrt{\log\tfrac1\delta }\big(1-\tfrac{\sqrt{2}}{2}\big) \sqrt{\delta} - \gamma  \sqrt{\log\tfrac1\delta} \cdot \big(2 \sqrt{2}+\tfrac{\sqrt{2}}{2}\big)^2 \delta
\\
&= \big[C \cons_1 \big(1-\tfrac{\sqrt{2}}{2}\big)
-\big(2 \sqrt{2}+\tfrac{\sqrt{2}}{2}\big)^2 \big]\gamma
\sqrt{\log\tfrac1\delta}\delta.
\end{align*}
Comparing this to \eqref{eq_b1} we obtain
\begin{align*}
	|F_\lambda(\mu)| \geq \tfrac{1}{2}\big[C \cons_1 \big(1-\tfrac{\sqrt{2}}{2}\big)
	-\big(2 \sqrt{2}+\tfrac{\sqrt{2}}{2}\big)^2 \big] |F_\lambda(\mu)-F_\lambda(\mu')|.
\end{align*}
Therefore, \eqref{st11} holds, provided that $\cons_1$ is chosen large enough.

In summary, we conclude that $\lambda \in \outseta$ and, since $|\lambda-\zeta|_\infty \leq \sdeltaa \delta$, we have that $\zeta \in \outseta + Q_{\sdeltaa \delta}(0)$. Hence, \eqref{claim1} holds, as claimed.

\subsection{Winding numbers compute charges}\label{sec_wn}
We show the following: for every $\zeta \in \{F=0\} \cap \domain$ and every $\lambda \in \outseta$ with
$|\lambda - \zeta|_\infty \leq \ccdelta$ we have
\begin{align}\label{isprecisely}
\theta_\lambda = \charge_\zeta \in \{-1,1\}.
\end{align}
Let $\zeta \in \{F=0\} \cap \domain$ and $\lambda \in \outseta$ with
$|\lambda - \zeta|_\infty \leq \ccdelta$. Note that we must actually have $|\lambda - \zeta|_\infty < \ccdelta$ because we have excluded the event where $F$ vanishes on $\partial Q_{\ccdelta}(\lambda)$.

The box $Q_{\ccdelta}(0)$ contains at least one zero of $F_\lambda$, namely $\zeta-\lambda$, and no other zeros due to \eqref{eq_7sep}, as, indeed, $\mathrm{diam} \,Q_{\ccdelta}(\lambda) = 2\sqrt{2} \sdelta\delta < 7 \sdelta\delta$ and $Q_{\ccdelta}(\lambda)\subset \Omega_{L+\sdelta\delta}$. In addition, recall that we have excluded the event that $F_\lambda$ has multiple zeros, cf. \eqref{eq_nmz}.
Therefore, by Poincar\'{e}'s index formula \eqref{eq_Poincare}, the winding number satisfies
\begin{align}\label{eq_wn}
\frac{1}{2\pi i}\int_{\partial Q_{\ccdelta}(0)} \frac{dF_\lambda}
{F_\lambda} =
\sgn \det D F_\lambda (\zeta-\lambda) \in \{-1,1\}.
\end{align}
On the other hand, since $\lambda \in \outseta$, \eqref{eq_arg} shows that winding number \eqref{eq_wn} is precisely $\theta_\lambda$. 
By definition,
$\charge_\zeta = \sgn \det D F(\zeta)$. To complete the proof of \eqref{isprecisely},
we show that
$ \det D F_\lambda (\zeta-\lambda)= \det D F (\zeta)$.

Denoting $\lambda=a+ib$ and $z=x+iy$, then $F_\lambda(z)=e^{i(bx-ay)}F(z+\lambda)$ and
\begin{align*}
\partial_x F_\lambda(z) &= e^{i(bx-ay)} \big[ 
\partial_x F(z+\lambda)+
ib F(z+\lambda)\big],
\\
\partial_y F_\lambda(z) &= e^{i(bx-ay)} \big[ \partial_y F(z+\lambda)-ia F(z+\lambda)\big].
\end{align*}
Thus, using the expression in \eqref{eq_jacv},
\begin{align*}
\det D F_\lambda(\zeta-\lambda)=
-\Im \big[\partial_x F_\lambda(\zeta-\lambda) \cdot \overline{\partial_y F_\lambda(\zeta-\lambda)}\big]
=-\Im \big[\partial_x F(\zeta) \cdot \overline{\partial_y F(\zeta)}\big]=\det DF(\zeta),
\end{align*}
as claimed.

\subsection{After Step 2, each true zero is close to a computed zero}

We show that 
\begin{align}\label{claim2}
\{F=0\} \cap \domain \subset \outsetb + 
Q_{\sdeltaa \delta}(0).
\end{align}
Let $\zeta \in \{F=0\} \cap \domain$. By \eqref{claim1}, there exists $\lambda \in \outseta$ such that $|\lambda-\zeta|_\infty \leq \sdeltaa \delta$.
By \eqref{isprecisely}, $\theta_\lambda \not= 0$. This means that $\lambda \in \outsetb$ and shows that $\zeta \in  Q_{\sdeltaa \delta }(\lambda) \subset \outsetb + 
Q_{\sdeltaa \delta}(0)$.

\subsection{After Step 2, each computed zero is close to a true zero}
We show that 
\begin{align}\label{claim3}
	\outsetb \subset \{F=0\} + Q_{\sdelta\delta}(0).
\end{align}
Let $\lambda \in \outsetb$. This means that $\lambda \in \outseta$ and $\theta_\lambda \not= 0$. By \eqref{eq_arg}, the winding number \eqref{eq_wn} is non-zero. So, 
by Poincar\'{e}'s index formula \eqref{eq_Poincare},
$F_\lambda$ must vanish at some point $\zeta \in Q_{\ccdelta}(0)=Q_{\sdelta\delta}(0)$. Since  $\lambda+\zeta \in \{F=0\}$, this proves \eqref{claim3}.

\subsection{Definition of the parametrization}
We now look into the sieving step and analyze the final output set $\outset$. The following argument is similar to the one in \cite{efkr24} and is repeated for completeness.

Given $\zeta \in \{F=0\} \cap \domain$ we claim that there exists $\lambda \in \outset$ such that $|\zeta - \lambda|_\infty \leq \ccdelta$. Suppose to the contrary that
\begin{align}\label{eq_l1}
|\zeta - \mu|_\infty > \ccdelta, \qquad \mu \in \outset.
\end{align}
By \eqref{claim2}, there exists $\lambda \in \outsetb$ such that
$|\zeta-\lambda|_\infty \leq 
\sdeltaa \delta \leq \ccdelta$. By \eqref{eq_l1}, $\outset \subsetneq \outset \cup \{\lambda\}$. We claim that $\outset \cup \{\lambda\}$ is $5\ccdelta$-separated with respect to $|\cdot|_\infty$. Since, by construction, $\outset$ is $5\ccdelta$-separated with respect to $|\cdot|_\infty$, it is enough to check that
\begin{align*}
|\mu-\lambda|_\infty \geq 5\ccdelta, \qquad \mu \in \outset.
\end{align*}
If $\mu \in \outset$, by \eqref{claim3}, there exist $\zeta' \in \{F=0\}$ such that $|\zeta'-\mu|_\infty \leq \ccdelta$. If $\zeta'=\zeta$, then $|\zeta-\mu|_\infty \leq \ccdelta$, contradicting \eqref{eq_l1}. Thus $\zeta \not= \zeta'$, while, $\zeta' \in \outset+Q_{\ccdelta}(0) \subset \Omega_{L+\ccdelta}$. Hence, we are allowed to use \eqref{eq_7sep} to conclude that
\begin{align*}
|\mu - \lambda|_\infty \geq |\zeta-\zeta'|_\infty - |\lambda-\zeta|_\infty - |\mu-\zeta'|_\infty \geq 7 \ccdelta - \ccdelta - \ccdelta = 5\ccdelta.
\end{align*}
In conclusion, the set $\outset \cup \{\lambda\}$ is $5\ccdelta$-separated:
\begin{align*}
\inf \big\{ |\mu-\mu'|_\infty : \mu, \mu' \in \outset \cup \{\lambda\}, \mu \not= \mu' \big\} \geq 5  \ccdelta,
\end{align*}
which contradicts the maximality of $\outset$. Thus, a point $\lambda \in \outset$ such that $|\zeta - \lambda|_\infty \leq \ccdelta$ must indeed exist. We choose any such point, and define  $\map(\zeta) = \lambda$.

\subsection{Properties of the parametrization}

By construction, the map $\map$ satisfies 
\begin{align}\label{eq_dist2}
|\map(\zeta)-\zeta|_\infty \leq \ccdelta \leq 2 \sqrt{\delta},
\end{align}
and therefore \eqref{eq_dist} is proved.

Let us show that $\map: \{F=0\} \cap \Omega_L \to \outset$ is injective. Assume that $\map(\zeta)=\map(\zeta')$. Then, by \eqref{eq_dist2},
\begin{align*}
|\zeta-\zeta'|_\infty \leq |\map(\zeta) - \zeta|_ \infty + |\map(\zeta') - \zeta'|_ \infty \leq 2 \ccdelta,
\end{align*}
so the separation bound \eqref{eq_7sep} implies that $\zeta=\zeta'$.

We now show that each numerical zero that is away from the boundary is in the image of $\map$. Assume that $\lambda \in \outset \cap \domainminus$ and use \eqref{claim3} to select a zero $\zeta \in \{F=0\}$ such that $|\zeta-\lambda|_\infty \leq \ccdelta$. Then 
$\zeta \in \domain$, and, by \eqref{eq_dist2},
\begin{align}
|\map(\zeta) - \lambda|_\infty \leq |\map(\zeta) - \zeta|_\infty + |\zeta-\lambda|_\infty \leq 2 \ccdelta.
\end{align}
As $\lambda, \map(\zeta) \in \outset$ and $\outset$ is $5\ccdelta$-separated 
with respect to $|\cdot|_\infty$
by construction, cf. \eqref{eq_5sep}, we conclude that $\lambda=\map(\zeta)$.

Finally, we note that winding numbers are accurately computed. Given $\zeta \in \{F=0\}\cap\Omega_L$, we denote $\lambda := \map(\zeta)$. By \eqref{eq_dist2}, $\zeta$ and $\lambda$ are as in the claim in Section \ref{sec_wn}, which gives \eqref{wellcomp}. 

This concludes the proof of Theorem \ref{th_main}. \qed

\section{Numerical Evaluation}
\label{sec_num}
We proceed to test \talgo\, on synthetic signals and assess its effectiveness in computing zeros of the short-time Fourier transform. We also illustrate numerically the advantage of using twisted shifts in this application, and provide numerical support for Theorem \ref{th_main}.

\subsection{Simulation of the input model}
The first step is to simulate the grid values of a function from the input model described in Section \ref{sec_input_model}. In principle, this could be done by simulating a large Gaussian vector corresponding to all the samples of the input function on the finite grid \eqref{eq_cgrid}. To speed up computations, we introduce a fast method that amounts to an approximate low-rank factorization of the corresponding covariance matrix. Concretely, we consider the relation \eqref{eq_FF} and discretize the signal $f=f^1+\sigma f^0$, with $f^1\colon \mathbb{R}\to\mathbb{C}$, and $f^0$ complex Gaussian white noise. We approximate
\begin{align*}
    F(\sqrt{\pi}u-i\sqrt{\pi}v) 
    & = 
    e^{i \pi u v} \int_{\mathbb{R}}f(t)\overline{g(t-u)} e^{-2 \pi i v t} \, dt
    \\
    & \approx
    e^{i \pi u v} \sum_{t\in \mathbb{Z}} \int_{[t\delta,(t+1)\delta)}f(s) \, ds \cdot \overline{g(\delta t-u)} e^{-2 \pi i v \delta t}
    \\
    & \approx
    e^{i \pi u v} \sum_{t\in \mathbb{Z}} [\delta f^1(t\delta) +  \sigma \sqrt{\delta} \xi_t]
    \cdot \overline{g(\delta t-u)} e^{-2 \pi i v \delta t}    
    \,,
\end{align*}
where $\xi_t, t\in\mathbb{Z}$ are independent standard complex normal variables. Second, we fix a domain length $L$, a lattice spacing $\delta$ with $L/\delta \in \mathbb{N}$, and let $g^L$ be the $2L$ periodic function that agrees with $g$ on $[-L,L]$. We discuss how to compute
$F$ at a lattice point $\lambda = \delta k + i \delta j \in \grid_L$ --- of course, when applying this to test the different algorithms to compute $\{F=0\} \cap \domain$, we may need to take $L$ slightly larger to account for boundary effects. Set $u=\tfrac{\delta k}{\sqrt{\pi}}$, $v=\tfrac{-\delta j}{\sqrt{\pi}}$ above and let us further approximate
\begin{align}
F(\delta k + i \delta j)
&\approx e^{-i \delta^2 k j} \sum_{t\in \mathbb{Z}} [\delta f^1(t\delta) +  \sigma \sqrt{\delta} \xi_t]
    \cdot \overline{g(\delta (t-k/\sqrt{\pi}))} e^{2 \sqrt{\pi} i \delta^2 jt} 
    \\
&\approx
e^{-i \delta^2 k j} 
\sum_{t=0}^{2L/\delta-1} [\delta f^1(t\delta) +  \sigma \sqrt{\delta} \xi_t]
    \cdot \overline{g^L(\delta (t-k/\sqrt{\pi}))} e^{2 \sqrt{\pi} i \delta^2 jt} =: c_{j,k}.
\end{align}
When $g$ decays fast and $L/\delta$ is moderately large, the effect of the last truncation is very mild (and standard practice in Gabor analysis). To approximately simulate input functions $F$ at grid points $\lambda=\delta k + i\delta j$, we simply simulate the vector $(c_{j,k})_{j,k=-L/\delta,\ldots,L/\delta-1}$. See \cite[Section 5.1]{efkr24} for additional support for these approximations. We carried out the computations with GNU Octave \cite{octave} and the LTFAT toolbox \cite{ltfat}. Code reproducing the experiments can be accessed at \url{https://github.com/gkoliander/phase_jumps}.

\subsection{Experiments}
First, we consider the setting of the STFT of complex white noise using as window the Hermite function of index 1, $g(t) = t e^{-\pi t^2}$, (normalization is immaterial for our considerations). For this window, the expected charge and number of zeros is available in \cite[Corollary~1.10 and 1.13]{hkr22}, which provides a useful benchmark. Our experiment shows that, despite missing theoretical guarantees, the \talgo\, algorithm performs better than the \calgo\, algorithm for larger values of $\delta$, and gives satisfactory results already at moderate data resolution.
In a second example, we show that also when adding a deterministic signal, the simulations using \talgo\, conform with the theoretical predictions of the expected charge (Proposition \ref{lemma_fi}).

As a third example, we illustrate the
practical role of the phase stabilizing factors in \algo: although, in 
theory, the criteria to compute zeros and charges are unaffected by the multiplicative phase term in \eqref{eq_ts}, 
we illustrate that, in the context of the short-time Fourier transform sampled on grids, rapid phase changes render the non-stabilized version of \algo\, impractical.

\subsection{Experiment 1}
We use \talgo\, and \calgo\, to calculate the zero set and charges of $100$ realizations of the STFT of complex white noise, using the window function $g(t) = t e^{-\pi t^2}$,
and simulated as explained above.
The simulations are performed with a small value of the grid spacing parameter $\delta_{\textrm{min}}=2^{-10}$ on the square $[-4,4]^2$, but only zeros in $[-2,2]^2$ are used, so as to minimize boundary effects. 
To analyze the dependency on $\delta$, we subsample the fine grid using $\delta = 2^n \delta_{\textrm{min}}$ for $n=0, \dots, 6$.
In Fig.~\ref{fig:empzeros}, we show the computed charge and number of zeros for selected values of $\delta$ within an ever increasing domain and compare the results with ensemble averages (theoretically expected charge and number of zeros). While for the minimum spacing $\delta=2^{-10}$ there is no qualitative difference between the two algorithms, differences in performance become clear at lower data resolution. \talgo\, performs well with a grid spacing of at least $\delta = 2^{-4}$, whereas \calgo\, 
performs much more poorly with $\delta = 2^{-7}$. A more detailed post-hoc analysis showed that this is mainly due to Step 1, where \calgo\, excludes too many zeros.
\begin{figure*}[tbp]
    \centering
    \begin{subfigure}[b]{0.47\textwidth}
    \includegraphics[width=0.98\textwidth]{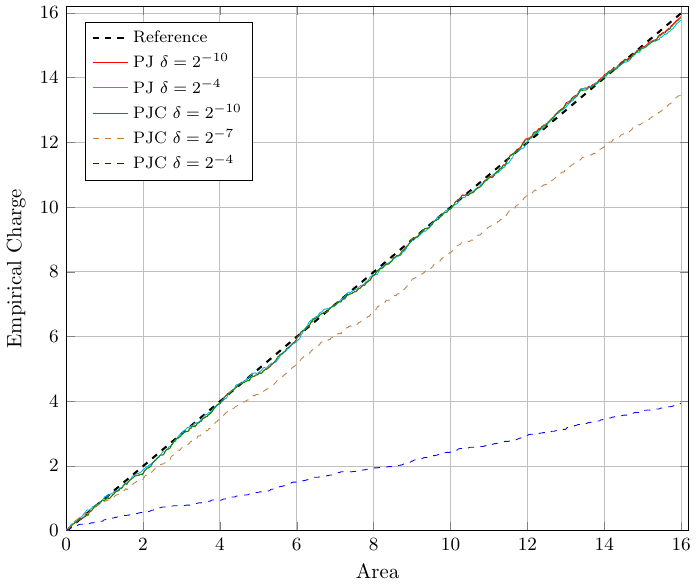}
    \end{subfigure}
    \hfill
    \begin{subfigure}[b]{0.47\textwidth}
    \centering
    \includegraphics[width=0.98\textwidth]{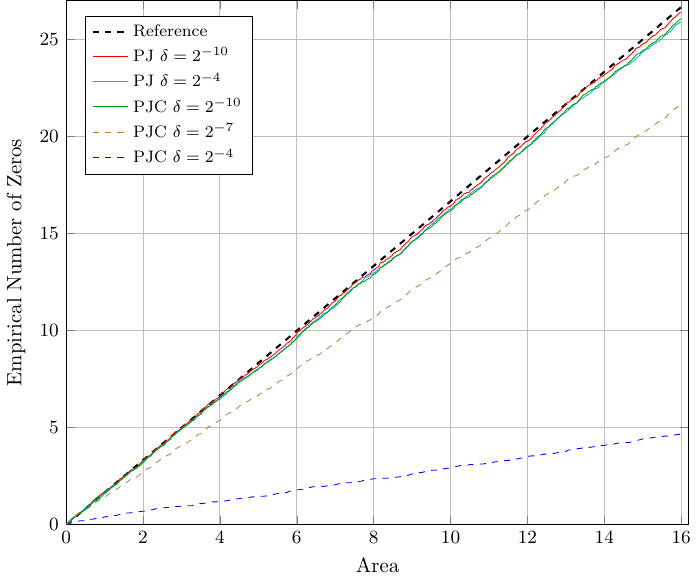}
    \end{subfigure}
    \caption{Empirical charge (left) and number of zeros (right) based on 100 simulations of the STFT of complex white noise with Hermite window $g(t) = t e^{-\pi t^2}$, using the algorithms \talgo\, (PJ) and \calgo\, (PJC) for different values of $\delta$.}
    \label{fig:empzeros}
\end{figure*}

\subsection{Experiment 2}
We use \talgo\, to calculate the zero set and charges of $500$ realizations of the STFT of the signal
$f^1(t) = 0.2\, e^{-\pi t^2}$ embedded in complex white noise, with window $g(t) = t e^{-\pi t^2}$. The simulations are performed with a  grid spacing parameter $\delta=2^{-5}$ on the square $[-4,4]^2$, but only zeros in $[-2,2]^2$ are used, so as to minimize boundary effects. 
Fig.~\ref{fig:empsigcharge} compares the observed charge within an ever increasing counting box to ensemble averages
provided by Proposition~\ref{lemma_fi}. In contrast to Experiment 1, due to the presence of a signal, the zero pattern is not stationary, and the particular choice of the counting boxes matters. We choose test boxes of increasing size
in a spiral pattern from the center, so as to first capture the domain that is heavily affected by the deterministic signal component. We expect to see a 
quick drop of the empirical charge to $-1$, corresponding to a negatively charged almost deterministic zero at the center, and then a smooth transition
to the full charge on $[-2,2]^2$, which is similar to the one obtained with no signal. The empirical curve follows the benchmark provided by Proposition~\ref{lemma_fi} very closely.
\begin{figure*}[tbp]
    \centering
    \begin{subfigure}[b]{0.47\textwidth}
    \includegraphics[width=0.98\textwidth]{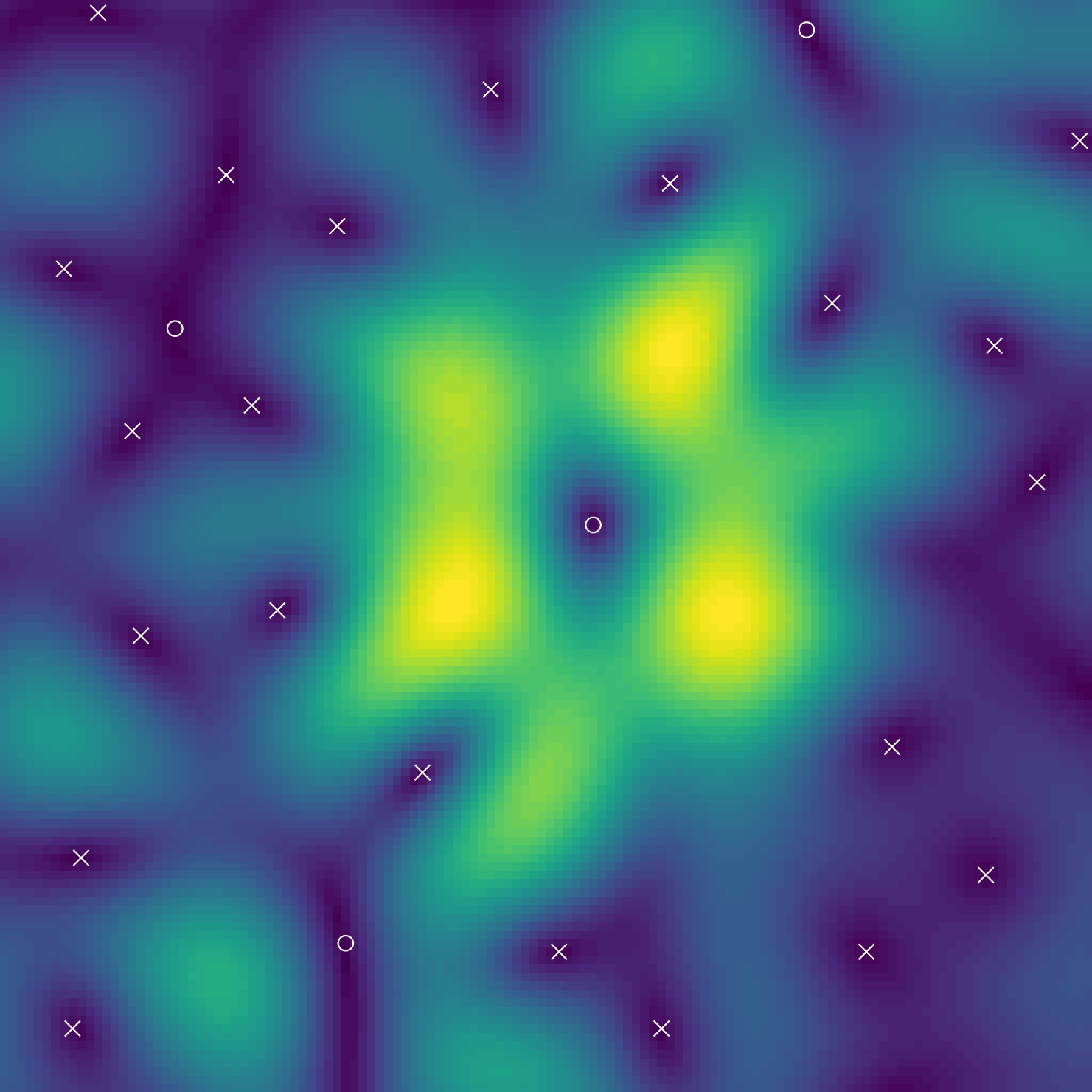}
    \end{subfigure}
    \hfill
    \begin{subfigure}[b]{0.47\textwidth}
    \centering
    \includegraphics[width=0.98\textwidth]{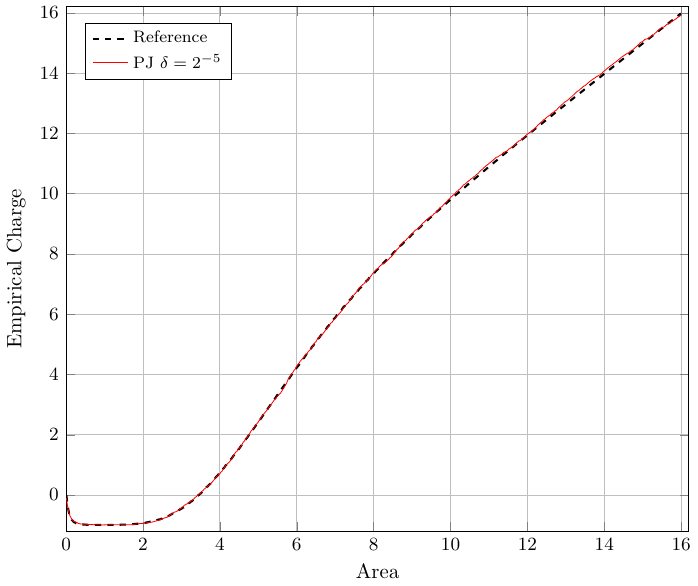}
    \end{subfigure}
    \caption{
    \textit{Left}: One simulation of the STFT of a Gaussian signal in additive complex white noise using the Hermite function of order 1 as window 
    (modulus with positive charges marked by $\times$ and negative charges marked by $\circ$). 
    \textit{Right}:  Empirical charge based on 500 simulations on $[-2,2]^2$ using the algorithms \talgo\, (PJ) with $\delta=2^{-5}$.}
    \label{fig:empsigcharge}
\end{figure*}

\subsection{Experiment 3}
We illustrate the importance of 
the phase-stabilization factors in applications to the short-time Fourier transform. Fig.~\ref{fig:phase} shows the magnitude and phase of a single simulation of the STFT of complex white noise with window $g(t)=te^{-\pi t^2}$. We note that phase varies very rapidly away from the origin, which could complicate computing winding numbers and make the detection of zeros less reliable.
\begin{figure*}[tbp]
    \centering
    \begin{subfigure}[b]{0.47\textwidth}
    \includegraphics[width=0.98\textwidth]{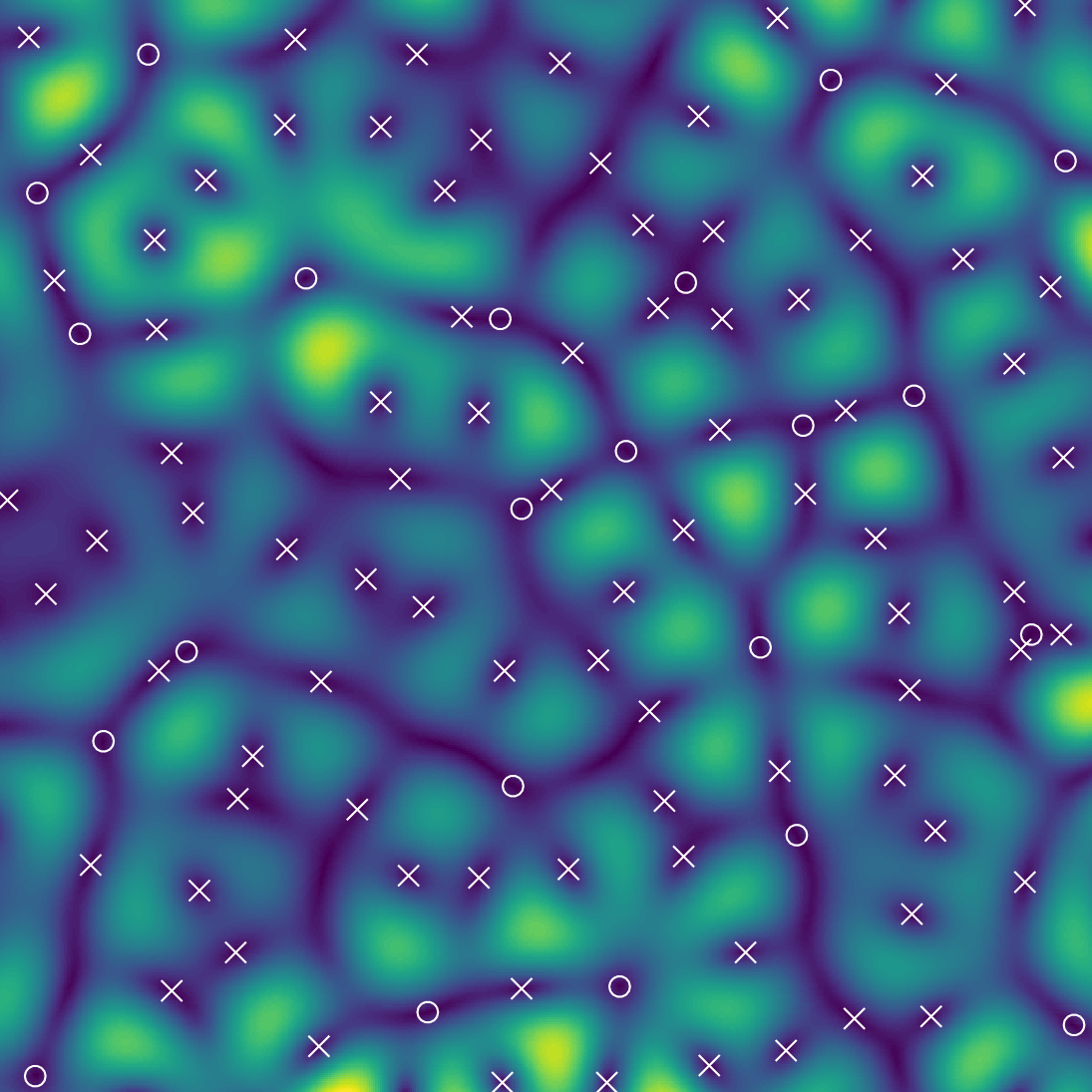}
    \end{subfigure}
    \hfill
    \begin{subfigure}[b]{0.47\textwidth}
    \centering
    \includegraphics[width=0.98\textwidth]{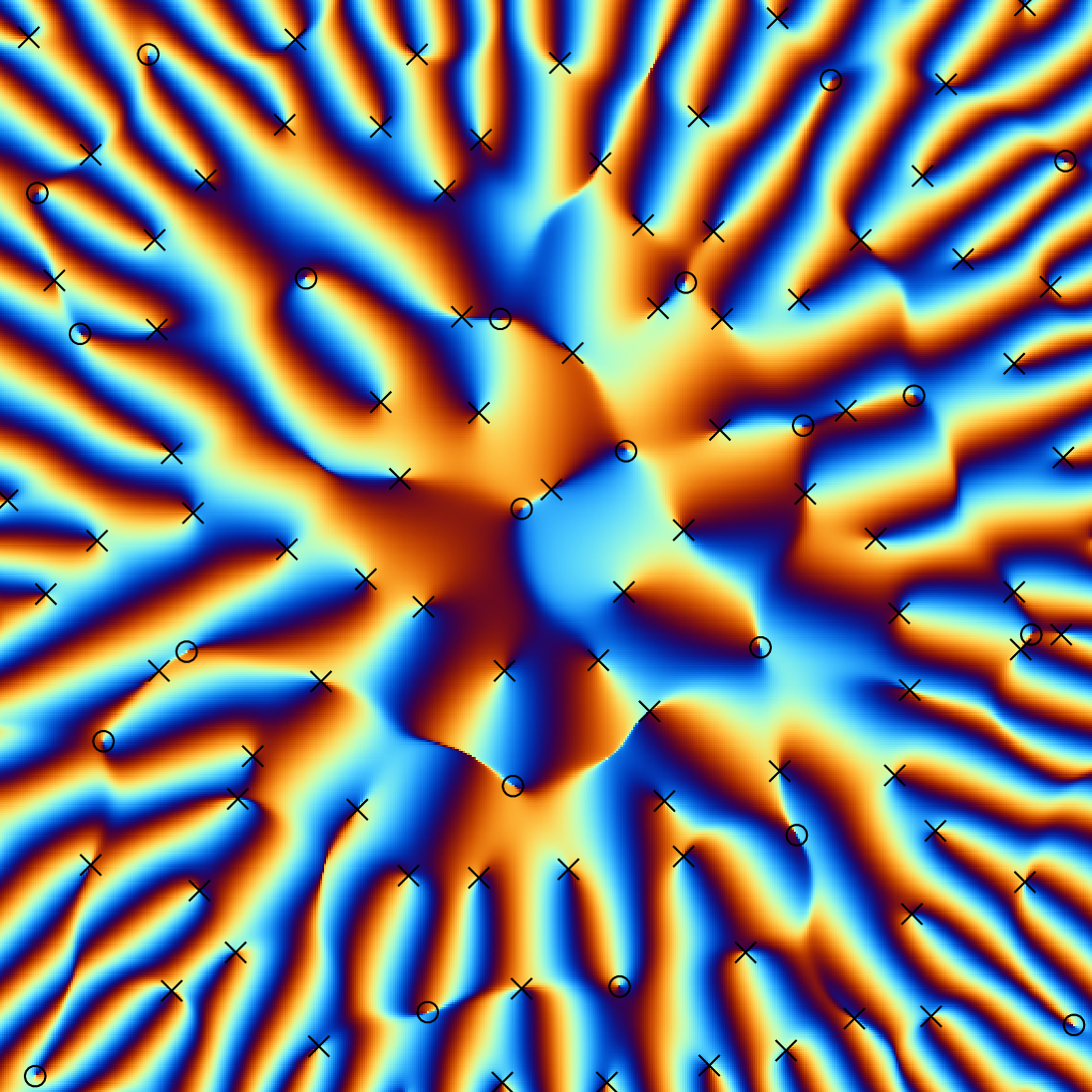}
    \end{subfigure}
    \caption{A simulation of the STFT of complex white noise using the Hermite function of order 1 as window (\textit{Left}: Modulus with positive charges marked by $\times$ and negative charges marked by $\circ$; 
    \textit{Right}: Phase). While the modulus is stationary, the phase varies strongly away from the origin (here at the center).}
    \label{fig:phase}
\end{figure*}
To formalize and quantify this intuition, we compare the performance of the Euclidean stationary and the twisted version of \algo\, on STFT data (which use, respectively, Euclidean or twisted shifts in \eqref{eq_ps}). 
We simulate $5$ realizations of
the STFT of complex white noise with window $g(t) = t e^{-\pi t^2}$  on the square $[-32,32]^2$ and with grid spacing $\delta = 2^{-5}$. As before, we only use the zeros in the smaller domain $[-16,16]^2$, so as to minimize boundary effects. We count zeros and charges on boxes 
of increasing area following a spiral pattern from the center. We then expect the stationary and twisted versions of the algorithm to perform similarly for small test boxes (because these contain points near the origin where the twisted shifts do not play a significant role). As the testing area increases, we expect both algorithms to perform differently. The results, depicted in Fig.~\ref{fig:empzeros_stat}, confirm this. The computation ignoring the phase stabilization deviates from ensemble averages already for mid-range values of the testing area, while the twisted version closely tracks the benchmark given by Proposition \ref{lemma_fi}.
\begin{figure*}[tbp]
    \centering
    \begin{subfigure}[b]{0.47\textwidth}
    \includegraphics[width=0.98\textwidth]{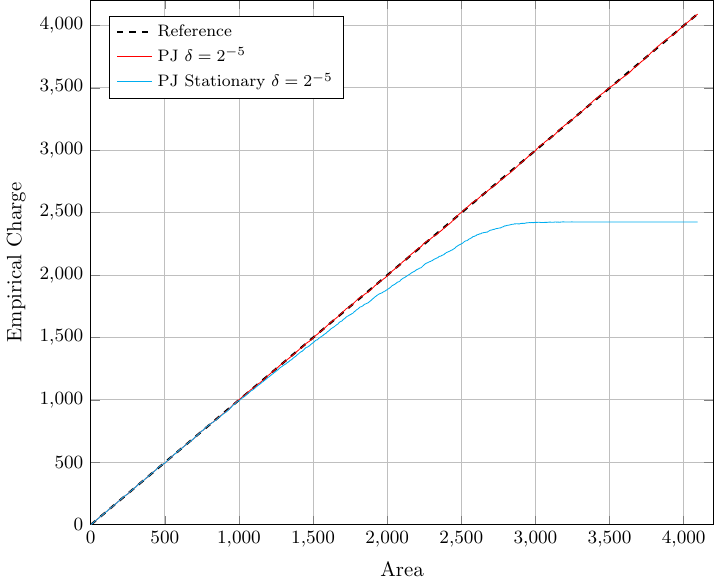}
    \end{subfigure}
    \hfill
    \begin{subfigure}[b]{0.47\textwidth}
    \centering
    \includegraphics[width=0.98\textwidth]{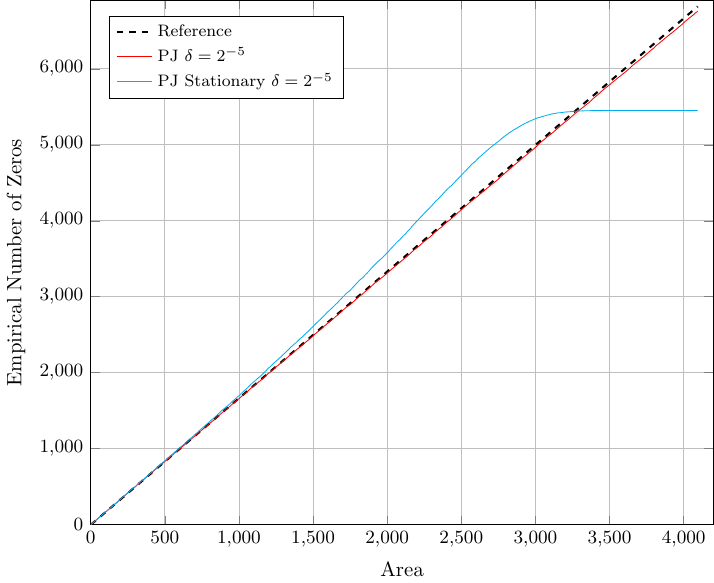}
    \end{subfigure}
    \caption{Empirical charge (left) and number of zeros (right) based on 5 simulations on $[-16,16]^2$ using the algorithms \talgo\, (PJ) and a Euclidean version (PJ Stationary) that does not use phase stabilization.}
    \label{fig:empzeros_stat}
\end{figure*}

\end{document}